\documentclass[12pt]{article}
\usepackage{amsmath,enumerate,amsfonts,color,amssymb, amsthm}

\def\ZZ{{\mathbb Z}}
\def\RR{{\mathbb R}}
\def\R{{\mathbb R}}

\def\CL{{\mycal{L}}}

\def\mcF{{\mycal F}}
\def\mcL{{\mycal L}}
\def\mcS{{\mycal S}}

\def\nd{\noindent}

\def\TCL{{\cal P}}
\def\sgn{{\rm sign}}

\newtheorem{theorem} {\sc  Theorem\rm} [section]

\newtheorem{lemma} [theorem] {\sc  Lemma\rm}
\newtheorem{proposition} [theorem] {\sc  Proposition\rm}

\newtheorem{definition}[theorem]{\sc  Definition\rm}
\newtheorem{open}[theorem]{\sc  Open problem\rm}
\newtheorem{remark}[theorem]{\sc  Remark\rm}

\def\nd{\noindent}

\newcounter{marnote}

\DeclareFontFamily{OT1}{rsfs}{}
\DeclareFontShape{OT1}{rsfs}{m}{n}{ <-7> rsfs5 <7-10> rsfs7 <10-> rsfs10}{}
\DeclareMathAlphabet{\mycal}{OT1}{rsfs}{m}{n}

\def\tr{{\rm tr}}

\def\mcS{{\mycal{S}}}

\def\Ss{{\mathbb S}}

\def\be{\begin{equation}}
\def\ee{\end{equation}}

\def\tr{{\rm tr}}

\def\mcS{{\mycal{S}}}

\def\mcE{{\mycal E}}

\def \f {\varphi}

\def\be{\begin{equation}}
\def\ee{\end{equation}}
\def\bea#1\eea{\begin{align*}#1\end{align*}}
\def\non{\nonumber}

\def\hu{{\hat u}}
\def\hv{{\hat v}}

\numberwithin{equation}{section}
\begin{document}

\title{Stability of point defects of degree $\pm \frac 1 2$ \\ in a two-dimensional nematic liquid crystal model}
\author{Radu Ignat\thanks{Institut de Math\'ematiques de Toulouse, Universit\'e
Paul Sabatier, 31062 Toulouse, France. Email: Radu.Ignat@math.univ-toulouse.fr
}~, Luc Nguyen\thanks{Mathematical Institute and St Edmund Hall, University of Oxford, Andrew Wiles Building, Radcliffe Observatory Quarter, Woodstock Road, Oxford OX2 6GG, United Kingdom. Email: luc.nguyen@maths.ox.ac.uk}~, Valeriy Slastikov\thanks{School of Mathematics, University of Bristol, University Walk, Bristol, BS8 1TW, United Kingdom. Email: Valeriy.Slastikov@bristol.ac.uk}~ and Arghir Zarnescu\thanks{University of Sussex, Department of Mathematics, Pevensey 2, Falmer, BN1 9QH, United Kingdom. Email: A.Zarnescu@sussex.ac.uk}\,\thanks{Institute of Mathematics ``Simion Stoilow" of the Romanian Academy, 21 Calea Grivitei Street, 01702 Bucharest, Romania}}

\maketitle
\begin{abstract}
We study $k$-radially symmetric solutions corresponding to topological defects of charge $\frac{k}{2}$ for integer $k \neq 0$ in the Landau-de Gennes model describing liquid crystals in two-dimensional domains. We show that the solutions whose radial profiles satisfy a natural sign invariance are stable when $|k| = 1$ (unlike the case $|k|>1$ which we treated before). The proof crucially uses the monotonicity of the suitable components, obtained by making use of the cooperative character of the system. A uniqueness result for the radial profiles is also established.
\end{abstract}


\section{Introduction}
We consider the Landau-de Gennes model describing nematic liquid crystals through functions taking values into the space $\mcS_0$ of the so-called $Q$-tensors:
$$\mcS_0=\bigg\{Q\in \R^{3\times3}\, :\,  Q=Q^t, \tr(Q)=0\bigg\},$$
 where $Q^t$ and $\tr(Q)$ are the transpose and the trace of $Q$.

We study critical points of the following Landau-de Gennes free energy functional:
\begin{align}\label{LDG}
\mcF(Q)= \int_{\Omega} \Big[\frac{1}{2}|\nabla{Q}|^2 + f_{bulk}(Q)\Big]\,dx, \quad Q \in H^1_{loc}(\Omega, \mcS_0),
\end{align} where 
$$
f_{bulk}(Q) = -\frac{a^2}{2}\tr(Q^2)-\frac{b^2}{3}\tr(Q^3)+\frac{c^2}{4}\left(\tr(Q^2) \right)^2,
$$
$a^2\geq 0$, $b^2, c^2>0$ and
$$\Omega=B_R\subset \R^2$$
is the disk centered at the origin of radius $R\in (0, \infty]$ (if $R=\infty$, then $\Omega=\RR^2$).
These critical points satisfy  the Euler-Lagrange system of equations:
\be\label{eq:EL}
 \Delta Q=-a^2 Q-b^2[Q^2-\frac 13|Q|^2I_3]+c^2|Q|^2\,Q \quad \textrm{in }\, \Omega,
\ee where $\frac 13|Q|^2=\frac 1 3 \tr(Q^2)$ is the Lagrange multiplier associated to the tracelessness constraint. Recall that every critical point $Q$ of $\mcF$ is smooth inside $\Omega$, see for instance \cite{Ma-Za}.

The main goal of this article  is to investigate the profile and energetic stability properties of certain symmetric solutions, the $k$-radially symmetric solutions, with $k=\pm 1$, that are physically relevant in describing the so-called ``point defect" patterns.

\begin{definition}
\label{def:k_rad}
Let $\Omega=B_R$ with $R\in (0, \infty]$. For $k \in \ZZ \setminus \{ 0 \}$, we say that a Lebesgue measurable map $Q:\Omega\to \mcS_0$ is $k$-{\it radially symmetric} if the following conditions hold for almost every $x=(x_1, x_2)\in \Omega$:
\vskip 0.5cm

\noindent {\bf (H1)} The vector $e_3 =(0,0,1)$ is an eigenvector of $Q(x)$.

\noindent {\bf (H2)} The following identity holds 
$$
Q\bigg(P_2\big({\mathcal R}_2 (\psi)  \tilde x\big)\bigg)= {\mathcal R}_k (\psi) Q(x) {\mathcal R}_k^t (\psi) , \ \textrm{for almost  every } \psi \in \RR, 
$$ where $\tilde x=(x_1,x_2,0)$, $P_2:\RR^3\to \RR^2$ is the projection given as $P_2(x_1,x_2,x_3)=(x_1,x_2)$ and
\begin{equation}
\label{eq: R_k }
{\mathcal R}_k (\psi) := \left(\begin{array}{ccc}\cos(\frac{k}{2}\psi) & -\sin(\frac{k}{2}\psi) & 0 \\\sin(\frac{k}{2}\psi) & \cos(\frac{k}{2}\psi) & 0 \\0 & 0 & 1\end{array}\right)
\end{equation} is the $\frac k 2$-winding rotation around the vertical axis $e_3$.
\end{definition}

\begin{remark}
In the previous work \cite{INSZ_AnnIHP}, we showed that if $k$ is an odd integer, then a map $Q \in H^1 (\Omega, \mcS_0)$ satisfying {\bf (H2)} automatically verifies {\bf (H1)} (see Proposition $2.1$ in \cite{INSZ_AnnIHP}). Therefore for the case $|k|=1$ the hypothesis {\bf (H2)} is sufficient.
\end{remark}

\vskip 0.5cm  We showed in \cite{INSZ_AnnIHP} that when $\Omega$ is a ball $B_R$ of radius $R \in (0, \infty]$ then $k$-radially symmetric solutions of \eqref{eq:EL} have a simple structure:
\be\label{anY}
Q(x)= u(|x|) \sqrt{2}\left(n(x)\otimes n(x)-\frac{1}{2}I_2\right) + v(|x|) \sqrt{\frac{3}{2}}\left(e_3\otimes e_3-\frac{1}{3}I_3\right), 
\ee
where the vector field $n$ is given by
\be\label{def:n}
n(r\cos \f, r\sin \f) =\left(\cos ({\textstyle\frac{k}{2}} \varphi) , \sin ({\textstyle\frac{k}{2}} \varphi) , 0\right), \quad r>0, \f\in [0, 2\pi),
\ee  $I_3$ is the $3\times 3$ identity matrix,
$I_2=I_3-e_3\otimes e_3$, 
 and $(u, v)$ satisfies on $(0, R)$ the following system of ODEs:
\be\label{ODEsystem}
\begin{cases}
u''+\frac{u'}{r}-\frac{k^2u}{r^2} &= h(u,v) \\
v''+\frac{v'}{r}&= g(u,v),
\end{cases}
\ee
with
\begin{align}
h(u,v)
	&= u\Big[-a^2+\sqrt{\frac{2}{3}} b^2 v+c^2\left( u^2+ v^2\right)\Big], \label{Eq:h=f_uDef}\\
 g(u,v)
 	&= v\Big[-a^2-\frac{1}{\sqrt{6}}b^2 v+c^2\left( u^2+ v^2\right) \Big] + \frac{1}{\sqrt{6} } b^2 u^2. \label{Eq:h=f_vDef}
 \end{align}

We couple the equation \eqref{eq:EL} with the boundary conditions that  are physically motivated and  compatible with the $k$-radial symmetry:
\be\label{BC1}
Q(x) = Q_k(x):=s_+ \left( n (x) \otimes n(x) -\frac{1}{3} I_3 \right) \quad \hbox{ as } x \in \partial B_R,
\ee
where the map $n:\overline\Omega\to \Ss^2$ is given by \eqref{def:n} and \be\label{def:s_+}
s_+=\frac{ b^2 + \sqrt{b^4+24 a^2 c^2}}{4 c^2}>0.
\ee
When $R = \infty$, equation \eqref{BC1} should be understood as
\[
\lim_{|x| \rightarrow \infty} |Q(x) - Q_k(x)| = 0.
\]
These boundary conditions also carry a topological information by having in a suitable sense a ``$\frac{k}{2}$ degree" for $n$, see the next subsection for details.
Moreover, the boundary condition \eqref{BC1}, together with the singular character of the ODE at the origin lead to the following  boundary conditions for the ODE system:
\be\label{bdrycond}
u(0)=0,{\ v'(0)=0}, \ u(R)=\frac{1}{\sqrt{2}} s_+,\,\,\,v(R)=-\frac{1}{\sqrt{6}}s_+.
\ee
(When $R=\infty$ we naturally define the boundary conditions in the limiting sense.)

The physical reasons  for the study of these solutions are given in the next subsection, that the mathematically-oriented reader may safely skip to reach the subsection detailing the main results.

\subsection{Physical background}
The $Q$-tensors describe the main characteristic feature of the nematic liquid crystal material, namely the local orientational ordering of the rod-like molecules and can be regarded as a crude measure of the local alignment (see \cite{dg, mottram2014introduction} for details). 

The simplest predictions are obtained by using Q-tensor valued maps in a free energy, whose minimizers describe equilibrium states. The type of free energy that we consider here is the simplest one that still captures fundamental physical aspects. The gradient part of the free energy density of a $Q$-tensor map, namely 
 $|\nabla Q|^2=\sum_{k=1}^2\sum_{i,j=1}^3 \big(\frac{\partial Q_{ij}}{\partial x_k}\big)^2 $ penalises the spatial variations while the bulk potential $f_{bulk}(Q)$  captures the specific liquid crystal aspects. It can be regarded as a Taylor-expansion (around the isotropic state $Q=0$) that respects the physical invariance $f_{bulk}(Q)=f_{bulk}(RQR^t)$ for $R\in SO(3)$ (see \cite{dg, mottram2014introduction} for details). The regime we consider (by choosing the sign  of the coefficient in front of $\tr(Q)^2$ in $f_{bulk}$)  is the deep nematic regime, in which case  the isotropic state $Q=0$ is an unstable critical point of the bulk potential. In general, thanks to suitable scalings \cite{mg,INSZ_AnnIHP}  one can physically think of the regime when $a^2,c^2$ fixed and $b^2\to 0$ as the ``low temperature regime", and we will use this terminology throughout the paper.
 
 Set
 \begin{equation}
 s_-=\frac{b^2 - \sqrt{b^4+24a^2c^2}}{4c^2} < 0.
 	\label{def:s_-}
 \end{equation}
 The bulk potential has two sets of local minima, namely,
 $$\Big\{s_-\left(n\otimes n-\frac{1}{3}Id\right), n\in\mathbb{S}^2\Big\}\textrm{ and } \Big\{s_+\left(n\otimes n-\frac{1}{3}Id\right), n\in\mathbb{S}^2\Big\},$$
  where the former set contains local minimizers while the latter one contains  all the global minimizers.
 
 We choose  the boundary conditions that are $k$-radially symmetric and belong to the set of global minimizers, as this allows for a direct comparison with the simpler director or Oseen-Frank theory and most importantly leads to a study 
 of liquid crystal defect profiles. Furthermore, one notes that the set 
\be \label{uniaxial}
\mcS_*=\left\{s_+\left(n\otimes n-\frac{1}{3}I_3\right)\, :\, n\in\mathbb{S}^2\right\}
\ee 
is homeomorphic to  $\RR P^2$ while the smaller set in which we consider boundary conditions, namely 
\be \label{s-lim}
\mcS_*^{lim}=\left\{s_+\left(n\otimes n-\frac{1}{3}Id\right), n\in \mathbb{S}^1\right\}
\ee
is homeomorphic to $\mathbb{R}P^1$. 
Moreover, if we consider $Q_k$  from \eqref{BC1} as an $\mathbb{R}P^1$-valued map on $\RR^2 \setminus \{0\}$, then it has degree $k/2$ about the origin. (For a definition of the degree for $\mathbb{R}P^1$-valued maps, see for instance \cite[pp. $685-686$]{BrezisCoronLieb}).

This model can be seen as the $2D$ reduction of the physical situation of a $3D$ cylindrical boundary domain, with so-called ``homeotropic" lateral boundary conditions
where the configurations are invariant in the vertical direction (see for instance \cite{bbch}).  Its main validation at a physical level is related to its capacity of describing certain patterns which provide the most striking optical signature of the liquid crystal and the very reason for the ``nematic" name (with nematic being related to a Greek word meaning ``thread"). These patterns are referred to as ``defect" patterns and are characterised by significant and {\it highly localised}  variations in the material properties. There are several types of defect patterns, the point defects being the simplest (see \cite{chandra, kleman, klelavr}). Nevertheless, despite their apparent simplicity the analytical investigation of their structure and profile generates  very  challenging nonlinear analysis  problems \cite{alama2015minimizers, BallZar, baumanphilips, canevari, contreras2014biaxial, DRSZ, Fat-Slas, golovaty2015dimension, INSZ_CRAS, ODE_INSZ, INSZ_AnnIHP, INSZ3, lamy}.

One can classify topologically the $2D$ point defects, by the topological degree of the so-called ``optical eigenvector" namely the eigenvector corresponding to the largest eigenvalue (assuming that this is also restricted to a plane). Thus the symmetric solutions we described are the prototypical types of defects, the most symmetrical such types of defects. 

There is a direct analogy with the Ginzburg-Landau theory of superconductivity, where the defects are also classified topologically and it is known that only the ``lowest degree" defects are stable, but not the higher degree ones, see \cite{Mironescu-radial}. The Ginzburg-Landau system exhibits a number of analogies with our case, however there are significant differences and additional difficulties in our case, see for instance the discussion in \cite{INSZ_CRAS, INSZ3}. We would like to remark that this analogy can be quite misleading in certain circumstances, for example, in the context of 3D Landau-de Gennes theory, the melting hedgehog  which is in a (debatable) sense  the ``lowest degree'' defect, can be unstable in a certain temperature regime \cite{mg, INSZ_CRAS}.

\subsection{Main mathematical results}

In \cite{INSZ_AnnIHP}, we constructed solutions $(u,v)$ of \eqref{ODEsystem} and \eqref{bdrycond} in $(0, R)$ with $R\in (0, \infty]$ using variational methods. These solutions give rise to $k$-radially symmetric solutions $Q$ of the 
Euler-Lagrange equations \eqref{eq:EL} with the boundary conditions \eqref{BC1} via \eqref{anY}. Furthermore, these solutions satisfy 
\[
u>0 \text{ and }v<0 \text{ in }(0,R)
\]
and they are local minimizers of the corresponding energy functional, in the sense that, for any $R' < R$ and any $(\xi, \eta) \in C_c^\infty(0,R')$ satisfying
\[
\sup_{(0,R')}|\eta| < \min\Big(\frac{s_+}{ \sqrt6}, \sqrt{\frac{2}{3}}|s_-|\Big)
\]
there holds
\[
\mcE_{R'}(u,v) \leq \mcE_{R'}(u + \xi, v + \eta),
\]
where  
\begin{align}
 \label{def:mcR}
\mcE_{R'}(u,v)
	&= \int_0^{R'} \bigg[ \frac{1}{2} \left( (u')^2+(v')^2+\frac{k^2}{r^2}u^2 \right) + f(u,v)\bigg]\,rdr,\\
f(u,v)
	& =-\frac{a^2}{2}(u^2 + v^2)+\frac{c^2}{4}\left(u^2+v^2\right)^2  -\frac{b^2}{3 \sqrt{6}}v(v^2 -3u^2).
\end{align}
(Note that 
$$h(p,q)= \frac{\partial f}{\partial p} (p,q) \text{ and }g(p,q)=\frac{\partial f}{\partial q}(p,q),$$
and so \eqref{ODEsystem} is the Euler-Lagrange equation for $\mcE_{R'}$.)

When $R < \infty$, we can allow $R'  = R$ in the above definition of local minimality. However, on the infinite domain $(0, \infty)$, the energy $\mcE_\infty(u,v)=\infty$, and therefore the local minimality property of $(u,v)$ should be understood as above with any $R' < R=\infty$.

Our main result is the stability of the critical points $Q$ on $B_R$ for $R\in (0, \infty]$, defined by \eqref{anY} corresponding to {\bf any} stable solutions $(u,v)$ as above, of the Landau-de Gennes energy $\mcF$ in the case $k = \pm 1$. For any solution $Q$ of \eqref{eq:EL} subjected to the boundary condition \eqref{BC1} we define the second variation $\CL[Q] (P)$ at $Q$ in direction $P \in C_c^\infty(B_{R'}, \mcS_0)$ ($R' < R$) as follows: 
\begin{align}
{\CL}[Q](P)&=\frac{d^2}{dt^2}\Big|_{t=0} \int_{B_{R'}} \Big\{\frac{1}{2}|\nabla(Q+tP)|^2 + f_{bulk}(Q + tP)\Big\}\,dx\non\\
&=\int_{B_R}\Big\{|\nabla P|^2-{a^2}|P|^2-2b^2\tr(P^2 Q)+{c^2}\left(|Q|^2|P|^2+2|\tr(QP)|^2\right)\Big\}\,dx. \label{Eq:CLDef-X}
\end{align}
This definition extends to $P \in H_0^1(B_R, \mcS_0)$ for $R\in (0, \infty]$ (recall that $H_0^1(\R^2, \mcS_0)=H^1(\RR^2, \mcS_0)$). 

A related issue is the stability of the ODE solution $(u,v)$ on $(0, R)$ for $R\in (0, \infty]$. The second variation for $\mcE_R$ at a solution $(u,v)$ of \eqref{ODEsystem} and \eqref{bdrycond} in direction $(\xi,\eta) \in C_c^\infty(0,R')$, ($R'<R$) is defined similarly as
\begin{align}
B(\xi,\eta) 
	&= \frac{d^2}{dt^2}\Big|_{t = 0} \mcE_{R'}(u + t\xi, v + t\eta)\non\\
	&= \int_{0}^R \bigg\{ | \xi'|^2+\frac{k^2}{r^2} |\xi|^2 + | \eta' |^2 + \left(-{a^2} + \frac{2b^2}{\sqrt{6}} v + c^2 (3u^2+v^2) \right) |\xi |^2 \non \\
		&\qquad\qquad + \left(-{a^2} - \frac{2b^2}{\sqrt{6}} v + c^2 (u^2+3 v^2) \right) |\eta |^2 
			 + 4 u \xi \eta \left( \frac{b^2}{\sqrt{6}} + c^2 v \right)\bigg\}\, rdr.\label{Eq:FD2EDef}
\end{align}
This definition extends to $(\xi,\eta) \in \hat X_R$, where $\hat X_R$ is the completion of $C_c^\infty(0,R)$ under the norm $$
\|(\xi,\eta)\|_{\hat X_R}^2 = \int_0^R \Big[|\xi'|^2 + |\eta'|^2 + (1 + r^{-2})|\xi|^2 + |\eta|^2\Big]\,r\,dr.
$$
In fact,
\begin{align*}
&\bullet \text{ if } R \in (0,\infty), \quad \hat X_R 
	= \Big\{(\xi,\eta): [0,R] \rightarrow \RR^2\,\Big| \sqrt{r}\xi', \sqrt{r}\,\eta', \frac{1}{\sqrt{r}}\xi, \sqrt{r}\eta \in L^2(0,R), \\
			&\qquad\qquad\qquad\qquad\qquad \xi(R) = \eta(R) = 0\Big\},\\
&\bullet \text{ if } R=\infty, \quad \hat X_\infty
	= \Big\{(\xi,\eta): [0,\infty) \rightarrow \RR^2\,\Big| \sqrt{r}\xi', \sqrt{r}\,\eta', \Big(\frac{1}{\sqrt{r}} + \sqrt{r}\Big)\xi, \sqrt{r}\eta \in L^2(0,\infty)\Big\}.
\end{align*}
We refer to Lemma \ref{Lem:Strauss} below for the behavior of $(\xi,\eta) \in \hat X_R$.

We recall our previous result from \cite{INSZ_AnnIHP} on the instability of $k$-radially symmetric solutions $Q$ in $\R^2$ for $|k|>1$:

\begin{theorem}[\cite{INSZ_AnnIHP}] 
Assume that \footnote{In \cite{INSZ_AnnIHP}, $a^2$ was assumed to be strictly positive. However, an inspection of the arguments therein allows an easy extension to the case $a^2 = 0$.} $a^2 \geq 0$, $b^2, c^2 > 0$ and $|k| > 1$. Let $(u,v)$ be {\bf any} solution of \eqref{ODEsystem} on $(0,\infty)$ under the boundary condition \eqref{bdrycond} (with $R = \infty$) such that $u > 0$ and $v < 0$ on $(0,\infty)$. Then the solution $Q$ of \eqref{eq:EL} on $\RR^2$ given by \eqref{anY} and satisfying the boundary condition \eqref{BC1} is unstable with respect to $\mcF$, namely there exists $P \in C_c^\infty(\RR^2,\mcS_0)$ such that $\CL[Q](P) < 0$.
\end{theorem}

We complete the study of $k$-radially symmetric critical points of $\mcF$ with the following stability result for $k = \pm 1$ in any disk $B_R$ with $R \in (0, \infty]$.
\begin{theorem}[{\bf Stability}]\label{Thm:Stab}
Assume that $a^2 \geq 0$, $b^2, c^2 > 0$ and $k = \pm 1$. Let $R \in (0, \infty]$ and $(u,v)$ be {\bf any} solution of \eqref{ODEsystem} on $(0,R)$ under the boundary condition \eqref{bdrycond} such that $u > 0$ and $v < 0$. Assume further that $(u,v)$ is stable with respect to $\mcE_R$, i.e.
\begin{equation}
B(\xi,\eta) \geq 0 \text{ for all } (\xi,\eta) \in \hat X_R.
	\label{Eq:ODEStab}
\end{equation}
Then the solution $Q$ of \eqref{eq:EL} on $B_R$ given by \eqref{anY} and satisfying the boundary condition \eqref{BC1} is stable with respect to $\mcF$, i.e. $\CL[Q] (P) \geq 0$ for all $P \in H_0^1(B_R, \mcS_0)$. 

Furthermore, $\CL[Q](P) = 0$ for some $P \in H^1_0(B_R,\mcS_0)$ if and only if, for some $(\xi_0,\eta_0) \in \hat X_R$ satisfying $B(\xi_0,\eta_0) = 0$ and some constants $\alpha$, $\beta$,
\begin{itemize}
\item either $R = \infty$ and
\[
P(x) = \xi_0(|x|) \sqrt{2}\big[n(x)\otimes n(x)-\frac{1}{2}I_2\big] + \eta_0(|x|) \sqrt{\frac{3}{2}}\big[e_3\otimes e_3-\frac{1}{3}I_3\big] + \alpha\frac{\partial Q}{\partial x_1}(x) + \beta \frac{\partial Q}{\partial x_2}(x),
\]
\item or $R < \infty$ and
\[
P(x) = \xi_0(|x|) \sqrt{2}\big[n(x)\otimes n(x)-\frac{1}{2}I_2\big] + \eta_0(|x|) \sqrt{\frac{3}{2}}\big[e_3\otimes e_3-\frac{1}{3}I_3\big],
\]
\end{itemize}
where $n(x)$ is given by \eqref{def:n}.
\end{theorem}

\begin{remark}
Loosely speaking, the second part of Theorem \ref{Thm:Stab} asserts that the kernel of $\CL[Q]$ is generated by the kernel of the second variation $B$ of $\mcE_R$ at $(u,v)$ and $span\{\partial_{x_1} Q, \partial_{x_2} Q\}$. 
\end{remark}

Two-dimensional point defects in the Landau-de Gennes framework have been studied for quite some time in the literature; see e.g. \cite{baumanphilips, canevari, cladiskle, DRSZ, Fat-Slas, GolovatyMontero, Zhang, INSZ_AnnIHP, KraVirga, KraljVirgaZumer} (and also \cite{DIO, IgnOtt} in micromagnetics). Our motivation came from the paper \cite{DRSZ} which concerns the extreme low-temperature regime ($b^2=0$). It was shown therein that there exists a unique global minimizer of the Landau-de Gennes energy which is $k$-radially symmetric and provides the description of the ground state profile of a point defect of index $k/2$. We followed this up in \cite{INSZ_AnnIHP} with the case $b^2 > 0$ and established the instability of entire $k$-radially symmetric solutions when $|k| > 1$. 

Different but related questions were considered on more general domains and for more general boundary conditions in  \cite{baumanphilips, canevari, GolovatyMontero,Zhang}. To put Theorem \ref{Thm:Stab} in perspective, we draw attention to \cite{baumanphilips, canevari, GolovatyMontero}. In \cite{baumanphilips}, the Landau-de Gennes energy was investigated for functions taking values into  a restricted three dimensional space of $Q$-tensors. It was shown that, in the case of small elastic constant, the minimizers of Landau-de Gennes energy exhibit behavior similar to those of Ginzburg-Landau energy \cite{vortices}, namely for boundary conditions of degree $k/2$ there are exactly $k$ vortices of degree $\pm 1/2$. In \cite{canevari, GolovatyMontero} the minimizers of the full Landau-de Gennes energy were studied under non-orientable boundary conditions (which in our setting amounts to $k$ being odd). It was shown that in the low temperature regime and in the case of small elastic constant the minimizer has only one vortex.

The proof of Theorem \ref{Thm:Stab} uses the type of framework we set up to treat the analogous problem of stability/instability of the melting hedgehog in three dimensions \cite{INSZ_CRAS, ODE_INSZ, INSZ3}. The first step of the proof entails a careful choice of basis decomposition for $\mcS_0$ so that the problem reduces, via Fourier decompositions, to an infinite set of partially coupled problems which involve functions of only one variable. In a loose sense, this can be viewed as some kind of partial separation of variables. The reduced problem for each Fourier mode is then treated using the so-called Hardy decomposition tricks together with certain qualitative properties of the profile functions $u$ and $v$. In particular, the following monotonicity result is of special importance in our proof.

\begin{theorem}[{\bf Monotonicity}]\label{Thm:Mon}
Assume that $a^2 \geq 0$, $b^2, c^2 > 0$ and $k \neq 0$. Let $R \in (0, \infty]$ and $(u,v)$ be {\bf any} solution of \eqref{ODEsystem} on $(0,R)$ under the boundary condition \eqref{bdrycond} such that $u > 0$ and $v < 0$. \\
$\bullet$ If $b^4 < 3a^2c^2$, then $u'>0$ and $v'>0$ in $(0,R)$. \\
$\bullet$ If $b^4 = 3a^2c^2$, $u' > 0$ in $(0,R)$ while $v \equiv -\frac{s_+}{\sqrt{6}}$.\\
$\bullet$ If $b^4 > 3a^2c^2$, then $u'>0$ while $v'<0$ in $(0,R)$. 
\end{theorem}

Regarding the assumption that $(u,v)$ is stable for $\mcE_R$ in Theorem \ref{Thm:Stab}, we note that the solution $(u,v)$ constructed in \cite{INSZ_AnnIHP} (for any given $a^2 \geq 0$, $b^2, c^2 > 0$) is a local minimizer and thus stable. In fact, for ``small $b$'', we have the following uniqueness and strict stability result.
\begin{theorem}[{\bf Uniqueness}]\label{Thm:Uniq}
Assume that $a^2, b^2, c^2 > 0$, $R \in (0, \infty]$ and $k \neq 0$. If $b^4 \leq 3a^2c^2$, under the assumption that $u > 0$ and $v < 0$,  there exists a unique solution of \eqref{ODEsystem} on $(0,R)$ under the boundary condition \eqref{bdrycond}. Furthermore such $(u,v)$ is strictly stable in the sense that $B(\xi,\eta) > 0$ for all $(\xi,\eta) \in \hat X_R$, $(\xi,\eta) \not\equiv 0$. In particular, $(u,v)$ is a local minimizer of $\mcE_R$.
\end{theorem}

The results above lead to the following open problem.

\begin{open}
Are solutions of \eqref{ODEsystem} and \eqref{bdrycond} (with or without the assumption that $u > 0$, $v < 0$) unique?
\end{open}

\begin{remark}\label{Rem:UniSharp?}
The case $b^4=3a^2c^2$ was proved earlier in \cite{INSZ_AnnIHP} using a different method. A careful mixture of the proof of Theorem \ref{Thm:Uniq} in Section \ref{sec:uniq} below and various estimates in \cite{INSZ_AnnIHP} shows that Theorem \ref{Thm:Uniq} continues to hold for $b^4 \leq \frac{75}{7} a^2\,c^2$. However, since this can be shown to be non-sharp and there is a distinctive difference between the case $b^4 > 3a^2c^2$ and the case $b^4 \leq 3a^2c^2$ (e.g. change of the monotone behaviour of $v$), we have chosen to keep the statement of the result as above. It remains an open question if uniqueness holds for all $a, b$ and $c$. 
\end{remark}

The rest of paper is structured as follows. In Section \ref{sec:monoton}, we prove the monotonicity of $u$ and $v$ assuming sign constraints $u>0$ and $v<0$. 
In Section \ref{sec:uniq}, we prove fine properties of functions $(\xi, \eta)\in \hat X_R$ (see Lemma \ref{Lem:Strauss}) and the uniqueness result of Theorem \ref{Thm:Uniq}. Section \ref{sec:stab} is devoted to the proof of the stability result in Theorem \ref{Thm:Stab}. Finally, in the appendix, we include a calculus lemma which is needed in the body of the paper.

\section{Monotonicity}
\label{sec:monoton}

In this section we prove monotonicity of solutions $(u,v)$ of the system  \eqref{ODEsystem}. 
Let us fix some $R \in (0,+\infty]$ and consider the ODE system \eqref{ODEsystem} on $(0,R)$ subjected to the boundary condition \eqref{bdrycond}.

Assume further that \footnote{The existence of such solution was proved in \cite{INSZ_AnnIHP}.}
\[
u > 0 \text{ and } v < 0 \text{ in } (0,R).
\]
We showed in \cite[Propositions 3.4, 3.5, 3.7]{INSZ_AnnIHP} that
\begin{equation}
0 < u < \frac{s_+}{\sqrt2} \text{ and } \min(- \frac{s_+}{\sqrt{6}}, \frac{2s_-}{\sqrt{6}}) < v < \max(- \frac{s_+}{\sqrt{6}}, \frac{2s_-}{\sqrt{6}}) \,\,  \text{ in } (0,R),
	\label{Eq:uvBox}
\end{equation}
and (see Step 3 of the proof of Proposition 3.1 in \cite{INSZ_AnnIHP})
\begin{equation}
\sqrt{3} v + u < 0 \text{ in }(0,R).
	\label{Eq:r3vu}
\end{equation}

The monotonicity of $u$ and $v$ depends on how big $b^4$ is relative to $a^2c^2$. When $b^4 = 3a^2c^2$, one has $v \equiv -\frac{s_+}{\sqrt{6}}$ and $u$ is the unique solution of the ODE $$u'' + \frac{u'}{r} - \frac{k^2}{r^2}u = c^2\,u(u^2 - \frac{s_+^2}{2}), u(0) = 0, s(R) = \frac{s_+}{\sqrt{2}}.$$
For other values of $b$, the monotonicity of $u$ and $v$ will be proved using the theory of cooperative systems and the moving plane method (see e.g. \cite{Sirakov07}). 

\begin{proof}[Proof of Theorem \ref{Thm:Mon}]
The case when $b^4 = 3a^2c^2$ is a consequence of \cite[Proposition 3.5]{INSZ_AnnIHP} and \cite[Lemma 3.7]{ODE_INSZ}. We assume for the rest of the proof that $b^4 \neq 3a^2c^2$.

\medskip
\noindent\underline{Case 1:} $b^4 < 3a^2c^2$. 

First assume that $R < \infty$. We note that (see \cite{INSZ_AnnIHP})
\[
\partial_q h(p,q) = \partial_p g(p,q) < 0 \text{ for all }0 < p < \frac{s_+}{\sqrt{2}} \text{ and } \min(- \frac{s_+}{\sqrt{6}}, \frac{2s_-}{\sqrt{6}}) < q < \max(- \frac{s_+}{\sqrt{6}}, \frac{2s_-}{\sqrt{6}}).
\]
This plays an important role in our argument below.

For $0 < s < R$, define
\[
u_s(r) = u(2s - r) \text{ and } v_s(r) = v(2s - r) \text{ for } \max(0,2s - R) < r < s.
\]

Note that $h(u(R),v(R)) = g(u(R),v(R)) = 0$ and recall that $0 < u < u(R)$ and $v < v(R)$ in $(0,R)$ (thanks to \eqref{Eq:uvBox}). In particular, the function $\hat u = u - u(R)$ satisfies
\[
\hat u'' + \frac{1}{r} \hat u' - \frac{k^2}{r^2}\,\hat u = \frac{k^2}{r^2} u(R) + h(u,v) - h(u(R),v(R)) \geq h(u,v) - h(u(R),v) = \xi\,\hat u
\]
for some function $\xi \in C[0,R]$. As $\hat u(R) = 0$ and $\hat u < 0$ in $(0,R)$, we deduce from the Hopf lemma (see e.g. \cite[Lemma 3.4]{GT}) that $u'(R) > 0$. Likewise, we can show that $v'(R) > 0$. Consequently, there is some small $\epsilon > 0$ such that $u_s > u$ and $v_s > v$ in $\max(0,2s - R) < r < s$ for any $R - \epsilon < s < R$. 

We define
\[
\underline{s} = \inf\Big\{0 < s < R : u_{s'} > u \text{ and }v_{s'} > v \text{ in }\max(0,2{s'} - R) < r < s' \text{ for all } s' \in (s,R)\Big\},
\]
then $\underline{s} \in [0,R)$.

We claim that $\underline{s}=0$. Assume by contradiction that $\underline{s}>0$, then,
\begin{enumerate}[(i)]
\item $u' \geq 0$ and $v' \geq 0$ in $(\underline{s},R)$,
\item and $u_{\underline{s}} \geq u > 0$ and $v_{\underline{s}} \geq v $ in $\max(0,2{\underline{s}} - R) < r < \underline{s}$.
\end{enumerate}
It follows that
\begin{align}
u_{\underline{s}}'' + \frac{1}{r}u_{\underline{s}}' - \frac{k^2}{r^2}\,u_{\underline{s}}
	&\leq h(u_{\underline{s}},v_{\underline{s}}) \leq h(u_{\underline{s}},v),\\
v_{\underline{s}}'' + \frac{1}{r}v_{\underline{s}}'
	&\leq g(u_{\underline{s}},v_{\underline{s}}) \leq g(u,v_{\underline{s}}) \text{ in } \max(0,2{\underline{s}} - R) < r < \underline{s},\label{Eq:PrepvsvDiff}
\end{align}
 and so
\begin{align}
(u_{\underline{s}}-u)'' + \frac{1}{r}(u_{\underline{s}}-u)' - \frac{k^2}{r^2}\,(u_{\underline{s}}-u)
	&\leq h(u_{\underline{s}},v) - h(u,v)=(u_{\underline{s}}-u)c_1(r),\label{Eq:usuDiff}\\
(v_{\underline{s}}-v)'' + \frac{1}{r}(v_{\underline{s}}-v)'
	&\leq g(u,v_{\underline{s}}) - g(u,v)=(v_{\underline{s}}-v)c_2(r),\label{Eq:vsvDiff}
\end{align}
with $c_1, c_2$ being two continuous functions on $[\max(0,2{\underline{s}} - R), \underline{s}]$.

Noting that, by \eqref{bdrycond} and \eqref{Eq:uvBox},
\begin{align*}
u_{\underline{s}}(\max(0,2{\underline{s}} - R)) 
	&> u(\max(0,2{\underline{s}} - R)),\\
u_{\underline{s}}(\underline{s}) 
	&= u(\underline{s}),
\end{align*}
we can appeal to the strong maximum principle and the Hopf lemma to conclude that
\begin{equation}
u_{\underline{s}} > u \text{  in }\max(0,2{\underline{s}} - R) < r < \underline{s} \text{ and } u_{\underline{s}}'(\underline{s}) > u'(\underline{s}).
	\label{Eq:usuStrict}
\end{equation}
This implies that the second inequality in \eqref{Eq:PrepvsvDiff} is strict and so is the inequality in \eqref{Eq:vsvDiff}. We again apply  the strong maximum principle and the Hopf lemma to obtain
\begin{equation}
v_{\underline{s}} > v \text{  in }\max(0,2{\underline{s}} - R) < r < \underline{s} \text{ and } v_{\underline{s}}'(\underline{s}) > v'(\underline{s}).
	\label{Eq:vsvStrict}
\end{equation}
Estimates \eqref{Eq:usuStrict} and \eqref{Eq:vsvStrict} contradict the minimality of $\underline{s}$. Therefore, $\underline{s} = 0$ as claimed. This proves that $u', v'\geq 0$ on $(0,R)$.

\bigskip

We turn to show that $u', v'> 0$ on $(0,R)$. We have the following equations for $(u', v')$:
\begin{align}\label{ODEsystem-der}
u'''+\frac{u''}{r}-\frac{(k^2+1)u'}{r^2} +  \frac{2 k^2u}{r^3}&={u'} \partial_u h(u,v) + v'\,\partial_v h(u,v),\nonumber\\
v'''+\frac{v''}{r}-\frac{v'}{r^2}&={v'} \partial_v g(u,v) + u'\,\partial_u g(u,v). 
\end{align}
Noting now that $\partial h(u,v) = \partial_u g(u,v) < 0$, we arrive at
\begin{align*}
u'''+\frac{u''}{r}-\frac{(k^2+1)u'}{r^2}
	&\leq {u'} \partial_u h(u,v) ,\nonumber\\
v'''+\frac{v''}{r}-\frac{v'}{r^2}
	& \leq {v'} \partial_v g(u,v) . 
\end{align*}
Since $u'(R) > 0$ and $v'(R) > 0$, the strong maximum principle implies that $u'>0$ and $v'>0$, as desired.

\bigskip

Next, consider the case $R = \infty$. In order for the above argument to carry through, we need to show that there is some $R_0 > 0$ such that
\[
u' > 0 \text{ and } v' > 0 \text{ in } (R_0,\infty).
\]
To this end, recall the asymptotics (see \cite{INSZ_AnnIHP})
\begin{align}
u (r)
	&= \frac{s_+}{\sqrt{2}} + p_1\,r^{-2} + O(r^{-4}),\label{Eq:uAsymp}\\
v(r)
	&= -\frac{s_+}{\sqrt{6}} + q_1\,r^{-2} + O(r^{-4}),\label{Eq:vAsymp}
\end{align}
where $p_1 = - \frac{\sqrt{2}k^2}{2}\,\frac{2b^2 + c^2\,s_+}{b^2(-b^2 + 4c^2\,s_+)}$, $q_1 = - \frac{\sqrt{6}k^2}{2}\,\frac{-b^2 + c^2\,s_+}{b^2(-b^2 + 4c^2\,s_+)}$, and the big `O' notation is meant for large $r$. In fact, the argument therein leads to an asymptotic expansion
\begin{align*}
u (r)
	&= \frac{s_+}{\sqrt{2}} + p_1\,r^{-2} + p_2\,r^{-4} + \ldots + p_N\,r^{-2N} +  O(r^{-2N -2}),\\
v(r)
	&= -\frac{s_+}{\sqrt{6}}  + q_1\,r^{-2} + q_2\,r^{-4} + \ldots + q_N\,r^{-2N} +  O(r^{-2N -2}),
\end{align*}
for any given $N \geq 2$ and with explicitly computable coefficients $p_i$, $q_i$'s. In particular, we obtain that
\begin{align}
\frac{1}{r}(ru')'  
	&= u'' + \frac{u'}{r} 
	= h(u,v) + \frac{k^2}{r^2}u = 4p_1\,r^{-4} + O(r^{-6}),\label{Eq:u'Asymp}\\
\frac{1}{r}(rv')' 
	&=v'' + \frac{v'}{r} 
	= g(u,v) = 4q_1\,r^{-4} + O(r^{-6}).\label{Eq:v'Asymp}
\end{align}
For $b^4 < 3a^2c^2$, we have $p_1 < 0$ and $q_1 < 0$. Hence, there is some $R_0 > 0$ such that $(ru')' < 0$ and $(rv')' < 0$ in $(R_0,\infty)$. Integrating twice, it follows that for any $R_0<s<r$, we have 
$$u(s)+su'(s)\log\frac r s\geq u(r), \quad v(s)+sv'(s)\log\frac r s\geq v(r).$$
As $u$ and $v$ have a limit as $r\to \infty$, this implies that $u' \geq 0$ and $v' \geq 0$ in $(R_0,\infty)$. Since 
$(ru')' < 0$ and $(rv')' < 0$ in $(R_0,\infty)$ we conclude that $u'>0$ and $v'>0$ in $(R_0, \infty)$. This completes the proof when $b^4 < 3a^2c^2$.

\bigskip
\noindent\underline{Case 2:} $b^4 > 3a^2c^2$. The proof is similar except the following changes:
\begin{itemize}
\item $\partial_v h(u,v) = \partial_u g(u,v) > 0$ for all $(u,v)$ satisfying \eqref{Eq:uvBox},
\item $v'(R) < 0$ if $R < \infty$,
\item $q_1 > 0$ if $R = \infty$.
\end{itemize}
We omit the details.
\end{proof}

\section{Uniqueness}\label{sec:uniq}

\def\tu{\tilde{u}}
\def\tv{\tilde{v}}

In this section, we prove Theorem \ref{Thm:Uniq}.

\vskip 0.2cm

{\bf Strategy}. In order to prove the uniqueness of the solution of the ODE system \eqref{ODEsystem}, \eqref{bdrycond}, for which $u>0$ and $v<0$, we follow a strategy similar to that in \cite{ABG-uniqueness}. We show that any solution of the ODE system with the mentioned signs of the components is a local minimizer if $b^4\leq 3a^2c^2$ (see Proposition \ref{Prop:ODEStability}). Then the uniqueness will follow by the mountain pass lemma. Indeed, assuming by contradiction that there exist two such solutions, we use a mountain pass argument to find another solution on a trajectory connecting the two given ones. This solution will have to be unstable thus leading to a contradiction and proving the uniqueness.
Many of the complications in our treatment are because of the lost of compactness due to the infinite domain (i.e. when $R = \infty$).

\medskip

{\bf Some notation}. Recall  
\begin{align*}
\mcE_R(u,v)& = \int_0^R \bigg[ \frac{1}{2} \left( (u')^2+(v')^2+\frac{k^2}{r^2}u^2 \right) + f(u,v)\bigg]\,rdr,
\end{align*}
where $$f(u,v) = -\frac{a^2}{2}(u^2 + v^2)+\frac{c^2}{4}(u^2+v^2)^2
 -\frac{b^2}{3 \sqrt{6}}v(v^2 -3u^2)$$
 and $$h(u,v)=\partial_u f(u,v), \quad  g(u,v)=\partial_v f(u,v).$$

For $0 < R < \infty$, if we define
\[
X_R = \Big\{(u,v): [0,R] \rightarrow \RR^2\,\Big| \sqrt{r}u', \sqrt{r}\,v', \frac{1}{\sqrt{r}}u, \sqrt{r}v \in L^2(0,R), u(R) = \frac{s_+}{\sqrt{2}}, v(R) = -\frac{s_+}{\sqrt{6}}\Big\},
\]
then $\mcE_R \in C^1(X_R,\RR)$. 

For $R = \infty$, we have a complication as $\frac{u^2}{r}$ and $r\,f(u,v)$ are not integrable over $(0,\infty)$. To fix this issue, it is useful to note that if $(u_1, v_1)$ and $(u_2,v_2)$ are two solutions of \eqref{ODEsystem}, \eqref{bdrycond}, then, thanks to the asymptotic estimate \eqref{Eq:uAsymp}, 
\[
\Big(\frac{1}{\sqrt{r}} + \sqrt{r}\Big)\,(u_1 - u_2) \in L^2(0,\infty).
\]

To accommodate both situations, we let $(u_0,v_0)$ be a {\bf fixed solution} of \eqref{ODEsystem} and \eqref{bdrycond}, satisfying $u_0 \geq 0$, $v_0 \leq 0$ (then by \cite{INSZ_AnnIHP} $u_0>0$, $v_0<0$ on $(0,R)$). Consider instead of $\mcE_R$ the modified functional for $R\in (0, \infty]$:
\begin{align*}
\hat\mcE_R(\hat u, \hat v) 
	&= \int_0^R  \frac{1}{2} \Big( |(u_0 + \hat u)'|^2 - |u_0'|^2 +|(v_0 +\hat v)'|^2 - |v_0'|^2\Big)\,r\,dr
		+ \int_0^R \frac{k^2}{2r}((u_0 + \hat u)^2 - u_0^2)\,dr  \\
			&\qquad\qquad + \int_0^R  \big[f(u_0 + \hat u,v_0 + \hat v) - f(u_0,v_0)\big]\,rdr,
\end{align*}
where $(\hat u, \hat v)$ belongs to $\hat X_R$ defined by
\begin{align*}
&\bullet \text{ if } R \in (0,\infty), \quad \hat X_R 
	= \Big\{(\hat u,\hat v): [0,R] \rightarrow \RR^2\,\Big| \sqrt{r}\hat u', \sqrt{r}\,\hat v', \frac{1}{\sqrt{r}}\hat u, \sqrt{r}\hat v \in L^2(0,R), \\
			&\qquad\qquad\qquad\qquad \hat u(R) = \hat v(R) = 0\Big\},\\
&\bullet \text{ if } R=\infty, \quad \hat X_\infty
	= \Big\{(\hat u,\hat v): [0,\infty) \rightarrow \RR^2\,\Big| \sqrt{r}\hat u', \sqrt{r}\,\hat v', \Big(\frac{1}{\sqrt{r}} + \sqrt{r}\Big)\hat u, \sqrt{r}\hat v \in L^2(0,\infty)\Big\}.
\end{align*}
It is clear that $\hat X_R$ is a Hilbert space with norm 
\[
\|(\hat u,\hat v)\|_{R}^2 = \int_0^R \Big[|\hat u'|^2 + |\hat v'|^2 + (1 + \frac{1}{r^2})\hat u^2 +\hat v^2\Big]\,r\,dr.
\]
It is clear that, for $R \in (0,\infty)$,
\[
\hat \mcE_R(\hu, \hv) = \mcE_R(u_0 + \hu, v_0 + \hv) - \mcE_R(u_0, v_0).
\]

We start with some basic remarks on the space $\hat X_R$ and the functional $\hat\mcE_R$.

\begin{lemma}\label{Lem:Strauss}
There is some constant $C > 0$ such that  for all $(\hu,\hv) \in \hat X_\infty$ we have
\begin{enumerate}
\item[1)] \textrm{Behaviour of $\hu$ and $\hv$ at $\infty$:}
\begin{equation}
|\hu(r)|^2 \leq \frac{C}{r} \int_{\frac12}^\infty [|\hu'|^2 + |\hu|^2]\,s\,ds \text{ and } |\hv(r)|^2 \leq \frac{C}{r} \int_{\frac12}^\infty [|\hv'|^2 + |\hv|^2]\,s\,ds, \quad 
\text{ for } r \in (\frac12,\infty).
	\label{Eq:StrIneql}
\end{equation}
In particular, $ \hu(r), \hv(r) \rightarrow 0$ as $r \rightarrow \infty$.
\item[2)] \textrm{Behaviour of $\hu$ in $(0, \infty)$:}
\[
|\hu(r)|^2 \leq  C \int_0^r [|\hu'|^2 + \frac1{s^2} |\hu|^2]\,s\,ds  \text{ for } r \in (0,\infty).
\]
In particular, $ \hu(r) \rightarrow 0$ as $r \rightarrow 0$.
\item[3)] \textrm{Behaviour of $\hv$ at the origin:}
\[
|\hv(r)|^2 \leq C\,|\ln r|\,\int_0^\infty [|\hv'|^2 + |\hv|^2]\,s\,ds \text{ for } r \in (0,\frac 1 2).
\]
\end{enumerate}
\end{lemma}

\begin{proof}
1) See the proof of Strauss inequality  \cite[p. 155]{Strauss77}.\\

2) For $0 < r < r_1$, we estimate
$$
|\hu^2(r_1)-\hu^2(r)| \leq 2\int_r^{r_1} |\hu(s)||\hu'(s)|\,ds\le 2\left(\int_r^{r_1} \frac{\hu^2(s)}{s}ds\right)^{\frac{1}{2}}\left(\int_r^{r_1} |\hu'(s)|^2\,s\,ds\right)^{\frac{1}{2}}.
$$
This implies that the limit $l:=\lim_{r\to 0} \hu^2(r)$ exists. Since $\int_0^\infty \frac{\hu^2(s)}{s}\,ds<\infty$ we thus have $l=0$, i.e. $\hu(r) \rightarrow 0$ as $r \rightarrow 0$. Returning to the above estimate, by the Young inequality we have $2|\hu(s)||\hu'(s)|\leq [|\hu'|^2 + \frac1{s^2} |\hu|^2]\,s$ 
and we obtain the desired estimate. \\

3) For $R > 0$, consider the minimization problem
\[
\alpha(R) = \inf \{\|v\|_{H^1((R,\infty);r\,dr)}: v \in H^1((R,\infty);r\,dr), v(R) = 1\Big\}.
\]
It is standard that the infimum is achieved and the minimizer $v_*$ is the unique solution of $v_*'' + \frac{1}{r}\,v_*' - v_* = 0$ in $(0,R)$, $v_*(R) = 1$, $v_*(\infty) = 0$. In terms of special functions, we have $v_* = \frac{K(r)}{K(R)}$, where $K$ is zeroth modified Bessel function of the second kind \cite{AbSt}. It is then readily seen that
\[
\alpha(R)^2 = \int_R^\infty [|v_*'|^2 + |v_*|^2]\,r\,dr = r\,v_*'\,v_*\Big|_R^\infty = \frac{R|K'(R)|}{K(R)}.
\]

As a consequence, we have for all $v\in H^1((0,\infty);r\,dr)$ that
\[
\|v\|_{H^1((0,\infty);r\,dr)} \geq \|v\|_{H^1((R,\infty);r\,dr)} \geq \alpha(R)\,|v(R)| \text{ for all } R \in (0,\infty).
\]
Now, since $K(r) = -\ln r + O(1)$ and $K'(r) = -\frac{1}{r} + O(1)$ as $r\rightarrow 0$ \cite{AbSt}, we have $\alpha(r) = \frac{1}{\sqrt{|\ln r|}} + O(1)$ as $r \rightarrow 0$, and so, for $r \le \frac{1}{2}$,
\[
|v(r)| \leq C\,\sqrt{|\ln r|}\,\|v\|_{H^1((0,\infty);r\,dr)}.
\]
\end{proof}

\begin{lemma}\label{Lem:Frechet} Assume that $R\in (0, \infty]$.\\
(1) $\hat \mcE_R: \hat X_R \rightarrow \RR$ is $C^1$ on $\hat X_R$ with the differential given by
\begin{align*}
 D\hat\mcE_R(\hu, \hv)(\xi,\eta)
	&= 
	\int_0^R \Big\{\hu'\,\xi' + \hv'\,\eta' + \frac{k^2}{r^2}\,\hu\,\xi \\
	&\qquad\qquad + [Df(u_0 + \hu, v_0 + \hv) - Df(u_0, v_0)](\xi,\eta)\Big\}\,r\,dr 
\end{align*}
where $(\hu, \hv), (\xi,\eta) \in \hat X_R$. Furthermore $(\hat u, \hat v) \in \hat X_R$ is a critical point for $\hat \mcE_R$ if and only if $(u, v) = (u_0 + \hat u, v_0 + \hat v)$ is a solution of \eqref{ODEsystem}, \eqref{bdrycond}.

\nd (2) $\hat\mcE_R$ is twice G\^ateaux differentiable, meaning here that for every $(\hu, \hv), (\xi,\eta) \in \hat X_R$, the following holds:
\begin{align*}
D^2\hat\mcE_R(\hu, \hv)(\xi,\eta)\cdot (\xi,\eta) 
	&:= \frac{d^2}{dt^2}\Big|_{t = 0} \hat\mcE_R(\hu + t\xi, \hv + t\eta)\\
	&=\int_0^R \Big\{|\xi'|^2 + |\eta'|^2 + \frac{k^2}{r^2}\,\xi^2  + D^2 f(u_0 + \hu, v_0 + \hv)(\xi,\eta) \cdot (\xi,\eta)\Big\}\,r\,dr.
\end{align*}
\end{lemma}

\begin{proof} The lemma is standard for $R < \infty$. Let us prove it for the more delicate case $R=\infty$.

\medskip
\noindent\underline{Step 1:} We prove that $\hat\mcE_\infty(\hu,\hv) < \infty$ for $(\hu, \hv) \in \hat X_\infty$. To this end, it suffices to prove the following three estimates (for some constant $C$):
\begin{align}
&\int_0^\infty [|u_0'|^2 + |v_0'|^2]\,r\,dr \leq C,\label{Eq:EiEst1}\\
&\int_0^\infty \frac{u_0\,|\hat u|}{r}\,dr \leq C\,\Big\{\int_0^\infty [|\hu'|^2 + |\hu|^2]\,r\,dr\Big\}^{1/2} \leq C \|(\hu,\hv)\|_{\hat X_\infty},\label{Eq:EiEst2}\\
&\int_0^\infty \big|f(u_0 + \hat u,v_0 + \hat v) - f(u_0,v_0)\big|\,r\,dr  \leq C(1 + \|(\hu,\hv)\|_{\hat X_\infty}^4). \label{Eq:EiEst3}
\end{align}

\medskip
\noindent\underline{Proof of \eqref{Eq:EiEst1}:} By \cite[Proposition 2.3]{INSZ_AnnIHP}, $u_0, v_0 \in C^2([0,\infty))$. In addition, by \eqref{Eq:u'Asymp} and \eqref{Eq:v'Asymp}, $r\,u_0'(r)$ and $r\,v_0'(r)$ have limits as $r \rightarrow \infty$. But as $u_0(r)$ and $v_0(r)$ remain finite as $r \rightarrow \infty$, these limits must be zero, i.e. 
\[
\lim_{r \rightarrow \infty } r\,u_0'(r) = \lim_{r \rightarrow \infty } r\,v_0'(r) = 0.
\]
Multiplying \eqref{Eq:u'Asymp} and \eqref{Eq:v'Asymp} by $r$ and integrating on $(r, \infty)$, we obtain
\begin{align}
u_0'(r) 
	&=  -2 p_1\,r^{-3} + O(r^{-5}),\label{Eq:u'AsX}\\
v_0'(r)
	&= -2q_1\,r^{-3} + O(r^{-5}).\label{Eq:v'AsX}
\end{align}
Therefore, $\sqrt{r} u_0'$ and $\sqrt{r}\,v_0'$ belongs to $L^2(0,\infty)$ and \eqref{Eq:EiEst1} follows.

\medskip
\noindent\underline{Proof of \eqref{Eq:EiEst2}:} By the Sobolev embedding theorem in one dimension, we have that $\hu$ and $\hv$ are continuous on $(0,\infty)$. Also, by Step $4$ in the proof of \cite[Proposition 2.3]{INSZ_AnnIHP}, $\frac{u_0(r)}{r^{|k|}}$ is bounded as $r \to 0$, and so
\begin{equation}
|u_0(r)|\le \frac{Cr^{|k|}}{(1+r)^{|k|}} \text{ for all } r \in (0,\infty).
	\label{Eq:u0Behavior}
\end{equation}
Estimate \eqref{Eq:EiEst2} is readily seen from Lemma \ref{Lem:Strauss}:
\begin{align*}
\int_0^\infty \frac{u_0\,|\hat u|}{r}\,dr&\leq C\int_0^{\frac12}|\hu|\, dr+C\int_{\frac12}^\infty \frac{|\hat u|}{r}\,dr\\
&\leq C\,\Big\{\int_0^\infty [|\hu'|^2 + |\hu|^2]\,r\,dr\Big\}^{1/2}\bigg(\int_0^{\frac12} \sqrt{|\ln r|}\, dr+ \int_{\frac12}^\infty \frac{1}{r^{3/2}}\,dr\bigg). 
\end{align*}

\noindent\underline{Proof of \eqref{Eq:EiEst3}:} First, note that $(\frac{s_+}{\sqrt{2}}, -\frac{s_+}{\sqrt{6}})$ is a (global) minimum of $f$. Thus there is some $\delta > 0$ and $C > 0$ such that for all $|x| \leq \delta$ and $|y| \leq \delta$, there holds
 \begin{equation}\label{Eq:fMinLoca}
0 \leq f\big(\frac{s_+}{\sqrt{2}} + x, -\frac{s_+}{\sqrt{6}} + y\big) - f\big(\frac{s_+}{\sqrt{2}}, -\frac{s_+}{\sqrt{6}}\big) \leq C(x^2 + y^2).
 \end{equation}
Therefore, in view of \eqref{Eq:StrIneql}, \eqref{Eq:uAsymp} and \eqref{Eq:vAsymp}, there is some sufficiently large $R_1 > 0$ such that 
\be\label{bulkinfty}
|f(u_0+\hu, v_0+\hv)-f(u_0,v_0)|\le C(\frac{1}{r^4}+\hu^2+\hv^2) \text{ in } (R_1, \infty),
\ee
 which implies
 \[
 \int_{R_1}^\infty |f(u_0+\hu, v_0+\hv)-f(u_0,v_0)|\,r\,dr \leq C(1 + \|(\hu,\hv)\|_{\hat X_\infty}^2).
 \]

On the other hand, since $f$ is a quartic polynomial and $u_0$ and $v_0$ are bounded, we have
\[
|f(u_0 + \hu, v_0 + \hv) - f(u_0,v_0)| \leq C(1 + \hu^4 + \hv^4) \text{ in } (0, \infty)
\]
for some constant $C$ which depends only on $a$, $b$ and $c$. Thus, by the Sobolev embedding theorem 
(in two dimensions $H^1(B_{R_1})\subset L^4(B_{R_1})$), 
\[
\int_0^{R_1} |f(u_0 + \hu, v_0 + \hv) - f(u_0,v_0)|\,r\,dr \leq C(1 + \|(\hu,\hv)\|_{\hat X_\infty}^4),
\]
which concludes the proof of \eqref{Eq:EiEst3}.

\medskip
\noindent\underline{Step 2:} We prove that $\hat\mcE_\infty$ is Fr\'echet differentiable.

Define $A: \hat X_\infty \rightarrow \RR$ by
\begin{align*}
A(\xi,\eta)
	&= 
	\int_0^\infty \Big\{\hu'\,\xi' + \hv'\,\eta' + \frac{k^2}{r^2}\,\hu\,\xi 	
	 + \big[Df(u_0 + \hu, v_0 + \hv) - Df(u_0, v_0)\big] (\xi,\eta)\Big\}\,r\,dr .
\end{align*}

Arguing as in the proof of \eqref{Eq:EiEst3}, we have 
\begin{equation}
|Df(u_0+\hu, v_0+\hv)-Df(u_0,v_0)|\le C(\frac{1}{r^2}+|\hu|+|\hv|) \text{ in } (R_1, \infty),
	\label{Eq:DfEstX1}
\end{equation}
(for possibly a larger constant $R_1$) and
\begin{equation}
|Df(u_0+\hu, v_0+\hv)-Df(u_0,v_0)|\le C(1 + |\hu|^3 + |\hv|^3) \text{ in } (0,\infty).
	\label{Eq:DfEstX2}
\end{equation}
As in Step 1, these estimates imply that $A$ is a well-defined and is continuous linear on $\hat X_\infty$, i.e., $|A(\xi,\eta)|\leq C \|(\xi,\eta)\|_{\hat X_\infty}$.

An easy computation shows that
\begin{align}
&\hat\mcE_\infty(\hu + \xi, \hv + \eta) - \hat \mcE_\infty(\hu, \hv) - A(\xi,\eta)\non\\
	&=  \int_0^\infty  \big[u_0'\xi' + \frac{k^2}{r^2} u_0\xi + h(u_0,v_0)\xi\big]\,r\,dr
		+ \int_0^\infty  \big[v_0'\eta'  + g(u_0,v_0)\eta\big]\,r\,dr \nonumber\\
		&\qquad + \int_0^\infty  P(\hu,\hv, \xi,\eta)\,r\,dr+ O(\|(\xi,\eta)\|_{\hat X_\infty}^2),
			\label{Eq:EiExcess}
\end{align}
where
\[
P(\hu,\hv, \xi,\eta)
	= f(u_0 + \hat u + \xi,v_0 + \hat v + \eta) - f(u_0 + \hu,v_0 + \hv) 
		- Df(u_0 + \hu,v_0 + \hv) (\xi,\eta).
\]

To treat the first two terms on the right hand side of \eqref{Eq:EiExcess}, recall that $u_0, v_0 \in C^2([0,\infty))$  \cite[Proposition 2.3]{INSZ_AnnIHP} and $v_0'(0) = 0$. In particular, $|v_0'(r)| \leq Cr$ for some constant $C$. Thus, using Lebesgue's dominated convergence theorem, \eqref{ODEsystem}, \eqref{Eq:u'AsX}, \eqref{Eq:v'AsX}, \eqref{Eq:u0Behavior}, \eqref{bdrycond} and Lemma \ref{Lem:Strauss} (in particular the behavior of $\xi(r)$ and $\eta(r)$ as $r \rightarrow 0$), we compute
\begin{align}
\int_0^\infty  \big[u_0'\xi' + \frac{k^2}{r^2} u_0\xi + h(u_0,v_0)\xi\big]\,r\,dr 
	&= \lim_{r \rightarrow \infty} u_0'(r)\,\xi(r) -  \lim_{r \rightarrow 0} u_0'(r)\,\xi(r) = 0,\label{Eq:EiExcess1}\\
\int_0^\infty  \big[v_0'\eta'  + g(u_0,v_0)\eta\big]\,r\,dr 
	&= \lim_{r \rightarrow \infty} v_0'(r)\,\eta(r) -  \lim_{r \rightarrow 0} v_0'(r)\,\eta(r) = 0.\label{Eq:EiExcess2}
\end{align}

To treat the remaining integral on the right hand side of \eqref{Eq:EiExcess}, we note that $f$ is a quartic polynomial, and so
\[
|P(\hu,\hv,\xi,\eta)|
	\leq C (\xi^2 + \eta^2)(1 + \xi^2 + \eta^2 + \hu^2 + \hv^2).
\]
Also, in view of Lemma \ref{Lem:Strauss} and the Sobolev embedding theorem, we have
\begin{align*}
\int_0^\infty [|\hu|^4 + |\hv|^4]\,r\,dr
	& \leq C\,\|(\hu,\hv)\|_{\hat X_\infty}^4,\\
\int_0^\infty [|\xi|^4 + |\eta|^4]\,r\,dr
	& \leq C\,\|(\xi,\eta)\|_{\hat X_\infty}^4.
\end{align*}
It follows that
\begin{equation}
\int_0^\infty  |P(\hu,\hv, \xi,\eta)|\,r\,dr = O\bigg(\|(\xi,\eta)\|_{\hat X_\infty}^2( 1 + \|(\xi,\eta)\|_{\hat X_\infty}^2)\bigg).
	\label{Eq:EiExcess3}
\end{equation}

Putting together \eqref{Eq:EiExcess},  \eqref{Eq:EiExcess1},  \eqref{Eq:EiExcess2} and  \eqref{Eq:EiExcess3}, we conclude that $\hat\mcE_\infty$ is Fr\'echet differentiable and $D\hat\mcE_\infty(\hu, \hv)= A$. 
Furthermore, since 
$$\int_0^R \Big\{u_0'\,\xi' + v_0'\,\eta' + \frac{k^2}{r^2}\,u_0\,\xi + Df(u_0, v_0)(\xi,\eta)\Big\}\,r\,dr=0 $$
for every $(\xi,\eta) \in \hat X_R$, we deduce that $(\hat u, \hat v) \in \hat X_R$ is a critical point for $\hat \mcE_R$ if and only if $(u, v) = (u_0 + \hat u, v_0 + \hat v)$ is a solution of \eqref{ODEsystem}, \eqref{bdrycond}.

\medskip
\noindent\underline{Step 3:} We prove that $\hat\mcE_\infty$ is twice G\^ateaux differentiable. 
 
Define $B: \hat X_\infty \rightarrow \RR$ by
\begin{align*}
B(\xi,\eta)
	&=\int_0^\infty \Big\{|\xi'|^2 + |\eta'|^2 + \frac{k^2}{r^2}\,\xi^2  + D^2 f(u_0 + \hu, v_0 + \hv)(\xi,\eta) \cdot (\xi,\eta)\Big\}\,r\,dr.
\end{align*}
The well-definedness of $B$ can be established similarly as in Step 2 using the estimate
\[
|D^2 f(u_0 + \hu, v_0 + \hv)| \leq C \text{ in some interval } (R_1, \infty)
\]
and
\[
|D^2 f(u_0 + \hu, v_0 + \hv)| \leq C(1 + |\hu|^2 + |\hv|^2) \text{ in }(0,\infty).
\]
Since
$$\int_0^{\frac12} |\hv|^2 |\eta|^2\, rdr\leq C \int_0^{\frac12} r|\ln r|^2\, dr \int_0^\infty [|\eta'|^2 + |\eta|^2]\,r\,dr ,$$ by Lemma \ref{Lem:Strauss},
we deduce that $B$ satisfies $|B(\xi,\eta)|\leq C \|(\xi, \eta)\|_{\hat X_\infty}^2$.
The assertion that
\[
B(\xi,\eta) = \frac{d^2}{dt^2}\Big|_{t = 0} \hat\mcE_\infty(\hu + t\xi, \hv + t\eta)=\frac{d}{dt}\Big|_{t = 0} D\hat\mcE_\infty(\hu + t\xi, \hv + t\eta) (\xi,\eta)
\]
follows from the estimate (for $0 < t < 1$)
\begin{multline*}
|Df(u_0 + \hu + t\xi, v_0 + \hv + t\eta) - Df(u_0 + \hu, v_0 + \hv) - t\,D^2f(u_0 + \hu, v_0 + \hv) (\xi,\eta)|\\ \leq 
	Ct^2(|\xi|^2 + |\eta|^2)(1 + |\hu| + |\hv| + |\xi| + |\eta|).
\end{multline*}
We omit the details.

\medskip
\noindent\underline{Step 4:} We prove the continuity of the differential $D\hat\mcE_\infty$ in $(\hu, \hv)\in \hat X_\infty$. 

Indeed, since the continuity is a local property, we may assume that $(\hu, \hv), (\tu, \tv)$ are in a finite ball of radius $\rho$ in $\hat X_\infty$. Then 
$$|Df(u_0 + \hu, v_0 + \hv)- Df(u_0 + \tu, v_0 + \tv)|\leq C(|\hu-\tu| + |\hv-\tv|)(1 + |\hu|^2 + |\hv|^2 + |\tu|^2 + |\tv|^2)$$
with $C>0$ independent of  $\hu, \hv, \tu, \tv$.  By Lemma \ref{Lem:Strauss}, we know that $\|\hu\|_\infty, \|\tu\|_\infty\leq C_\rho$ and
$$|\hv(r)|, |\tv(r)|\leq C_\rho\bigg(|\ln r|^{1/2} {\bf 1}_{(0, \frac 1 2)}(r)+\frac 1 r  {\bf 1}_{(\frac 1 2, \infty)}(r)\bigg), \quad r\in (0, \infty)$$
and therefore,
\begin{align*}
&\int_0^\infty |Df(u_0 + \hu, v_0 + \hv)- Df(u_0 + \tu, v_0 + \tv)|^2\, rdr\\
&\quad \leq C_\rho \|(\hu-\tu, \hv-\tv)\|_{\hat X_\infty}^2 (1 +\int_0^{\frac12} |\ln r|^4\,  rdr+\int_{\frac12}^\infty \frac1{r^3}\, dr).
\end{align*}
We conclude that for every $(\xi, \eta)\in \hat X_\infty$
$$\bigg|[D\hat\mcE_\infty(\hu, \hv)-D\hat\mcE_\infty(\tu, \tv)](\xi, \eta)\bigg|\leq C_\rho \|(\hu-\tu, \hv-\tv)\|_{\hat X_\infty}  \|(\xi, \eta)\|_{\hat X_\infty},$$
therefore $D\hat\mcE_\infty$ is locally Lipschitz in $\hat X_\infty$.
\end{proof}

Next, we consider coercivity and Palais-Smale properties of $\hat\mcE_R$.

\begin{lemma}\label{Lem:Coercivity}
 The following statements hold: \\
$\bullet$ if $R\in (0, \infty)$, then $\hat\mcE_R$ is coercive on $\hat X_R$;\\
$\bullet$ if $R=\infty$, then $\hat\mcE_\infty$ is coercive on the closed convex set 
\begin{equation}
M_\infty = \{(\hu, \hv) \in \hat X_\infty: u_0 + \hu \geq 0, v_0 + \hv \leq 0\},
	\label{Eq:MiCone}
\end{equation}
i.e. there exists some $C > 0$ such that
\[
\hat\mcE_\infty(\hu, \hv) \geq \frac{1}{C}\|(\hu, \hv)\|_{\hat X_\infty}^2 - C \text{ for all } (\hu,\hv) \in M_\infty.
\]
\end{lemma}

\begin{proof} Let $(\hu, \hv) \in \hat X_R$. In the argument below, $C$ denotes various positive constants which are always independent of $(\hu, \hv)$. 

Let $u_* = \frac{s_+}{\sqrt{2}}$ and $v_* = -\frac{s_+}{\sqrt{6}}$. Since $(u_*,v_*)$ is a minimum of $f$, we have $f(u_*,v_*) \leq f(x,y)$ for all $(x,y) \in \RR^2$. Also, by \eqref{Eq:fMinLoca}, \eqref{Eq:uAsymp} and \eqref{Eq:vAsymp} in the case $R=\infty$, 
\begin{equation}
\int_0^R \big|f(u_0 ,v_0 ) - f(u_*, v_*)\big|\,r\,dr 
	\leq C,
		\label{Eq:CostOfRep1}
\end{equation}
(if $R<\infty$ the above inequality is obvious since $u_0$ and $v_0$ are bounded)
which implies that
\begin{align*}
\int_0^R  \big(f(u_0 + \hat u,v_0 + \hat v) - f(u_0,v_0)\big)\,r\,dr  
	&\geq \int_0^R \big(f(u_0 + \hat u,v_0 + \hat v) - f(u_*, v_*)\big)\,r\,dr- C\\
	&\geq - C.
\end{align*}
\medskip
\noindent\underline{Case 1:} $R<\infty$.
From the above estimate, estimates on the finite domain analogous to \eqref{Eq:EiEst1}, \eqref{Eq:EiEst2}, and the Poincar\'e inequality in the disk $B_R$ for $\hu, \hv\in H^1_0(B_R)$:
$$\int_0^R |\hu|^2\, rdr\leq C \int_0^R |\hu'|^2\, rdr, \quad \int_0^R |\hv|^2\, rdr\leq C \int_0^R |\hv'|^2\, rdr,$$
we deduce that $$\hat\mcE_R(\hu, \hv)\geq C_1 \|(\hu, \hv)\|_{\hat X_R}^2 - C_3$$
which entails the coercivity of $\hat\mcE_R$ on $\hat X_R$.

\medskip
\noindent\underline{Case 2:} $R=\infty$. Due to the failure of Poincar\'e inequality in $H^1(\R^2)$, the above method does not work for $R=\infty$.  We conjecture that $\hat\mcE_\infty$ is not coercive on $\hat X_\infty$ and therefore, we prove coercivity only in $M_\infty$. Fix $(\hu, \hv) \in M_\infty$. We would like to improve the estimate on the integral of $|f(u_0 + \hat u,v_0 + \hat v) - f(u_*, v_*)|$. Let $Q = \{(x,y)\in\R^2, x\geq 0,y \leq 0\}$. By Lemma \ref{Lem:CalcLem} in the appendix, $f(u_*, v_*) < f(x,y)$ for all $(x,y) \in Q \setminus (u_*,v_*)$ and $D^2f(u_*,v_*)$ is positive definite. Also, $\frac{f(x,y)}{x^2 + y^2} \rightarrow \infty$ as $x^2 + y^2 \rightarrow \infty$. Thus, there is some positive constant $\alpha > 0$ such that
\[
f(x,y) - f(u_*,v_*) \geq \alpha((x - u_*)^2 + (y - v_*)^2) \text{ for all } (x,y) \in Q.
\]
This implies that
\begin{equation}
\int_0^\infty \big(f(u_0 + \hat u,v_0 + \hat v) - f(u_*,v_*)\big)\,r\,dr \geq \alpha\int_0^\infty \big[ |\hu + u_0 - u_*|^2 + |\hv + v_0 - v_*|^2\big]\,r\,dr.
	\label{Eq:MainCoer}
\end{equation}
Also, in view of \eqref{Eq:uAsymp} and \eqref{Eq:vAsymp},
\begin{equation}
\int_0^\infty \big[|u_0 - u_*|^2 + |v_0 - v_*|^2\big]\,r\,dr
	\leq C.
		\label{Eq:CostOfRep2}
\end{equation}
From \eqref{Eq:CostOfRep1}, \eqref{Eq:MainCoer} and \eqref{Eq:CostOfRep2}, we obtain
\begin{equation}
\int_0^\infty \big(f(u_0 + \hat u,v_0 + \hat v) - f(u_0,v_0)\big)\,r\,dr
	\geq \frac{1}{C} \int_0^\infty [\hu^2 + \hv^2]\,r\,dr -C. \label{Eq:CoerEst1}
\end{equation}
The desired coercivity of $\hat\mcE_\infty$ is now readily seen from \eqref{Eq:CoerEst1}, \eqref{Eq:EiEst1} and \eqref{Eq:EiEst2}.
\end{proof}

We recall that a continuously Fr\'{e}chet differentiable functional $I$ defined on a Banach space $X$ (i.e. $I \in  C^1(X, \RR)$) is said to satisfy the Palais-Smale condition if every sequence $\{u_n\}_{n=1}^\infty \subset X$ satisfying $\{I(u_n)\}_{n=1}^\infty$ is bounded and $D\, I (u_n) \to 0$ in $X'$ (dual of $X$)
is precompact in $X$, see e.g. \cite{Struwe}.

It is not difficult to prove that $\hat\mcE_R$ satisfies the Palais-Smale condition on $\hat X_R$ for finite $R$. It is not clear if this is the case for $R = \infty$; note the restricted coercivity we obtain in Lemma \ref{Lem:Coercivity}. We however content ourselves with a milder notion which suffices for our purpose and will be described in the sequel.

\begin{definition}[{\cite[Section II.12]{Struwe}}]
Let $X$ be a Banach space, $I \in C^1(X,\RR)$, and $M$ be a closed convex subset of $X$.
\begin{enumerate}[(i)]
\item We say that $x \in M$ is a critical point of $I$ relative to $M$ if
\[
\varrho(x) := \sup_{y \in M, \|y - x\|_X \leq 1} D\,I(x) ( x - y ) = 0.
\]
\item We say that $I$ satisfies the Palais-Smale condition on $M$ if every sequence $\{x_n\}_{n=1}^\infty \subset M$ satisfying $\{I(x_n)\}_{n=1}^\infty $ is bounded and $\varrho (x_n) \to 0$
is precompact in $X$.
\end{enumerate}
\end{definition}

\begin{lemma}\label{Lem:ConeCrit}
Let $M_\infty$ be as in \eqref{Eq:MiCone} and define 
\begin{equation}
\varrho(\hu, \hv) = \sup_{(\xi,\eta) \in M_\infty, \|(\xi -\hu, \eta-\hv) \|_{\hat X_\infty} \leq 1 }
 D\,\hat\mcE_\infty (\hu, \hv)(\hu - \xi, \hv -\eta).
 	\label{Eq:varrhoDef}
\end{equation}
If $(\hu_m, \hv_m) \in M_\infty$ converges weakly in $\hat X_\infty$ to $(\hu,\hv)$ and if
$\varrho(\hu_m, \hv_m)  \rightarrow 0$ as $m \rightarrow \infty$, then $(u_0 + \hu, v_0 + \hv)$ satisfies \eqref{ODEsystem}. In particular, $(\hu,\hv) \in M_\infty$ is a critical point of $\hat \mcE_\infty$ relative to $M_\infty$ if any only if $(u_0 + \hu, v_0 + \hv)$ satisfies \eqref{ODEsystem} and \eqref{bdrycond}.
\end{lemma}

\begin{proof}
It is enough to show the first assertion. In view of the Sobolev embedding theorem in one dimension, we can assume without loss of generality that $(\hu_m, \hv_m)$ converges uniformly on compact subsets of $(0,\infty)$ to $(\hu, \hv)$. 

Let $\varrho_m = \varrho(\hu_m,\hv_m)$. By definition, we have
\begin{equation}
D\hat\mcE_\infty (\hu_m,\hv_m)( \hu_m - \tilde u, \hv_m - \tilde v) \leq \varrho_m \text{ for all } (\tilde u,\tilde v) \in M_\infty, \|(\tilde u - \hu_m, \tilde v - \hv_m)\|_{\hat X_\infty} \leq 1. 
	\label{Eq:rho=0}
\end{equation}

In particular, we have
\[
D\hat\mcE_\infty (\hu_m,\hv_m)( \xi, \eta) \leq \varrho_m\|(\xi,\eta)\|_{\hat X_\infty} \text{ for all } (\xi,\eta) \in \hat X_\infty: \xi \leq 0, \eta \geq 0 \text{ in } (0,\infty).
\]
Since $(\hu_m, \hv_m)$ converges weakly to $(\hu, \hv)$ and $\varrho_m \rightarrow 0$, we deduce that
\[
D\hat\mcE_\infty (\hu,\hv)( \xi, \eta) \leq 0 \text{ for all } (\xi,\eta) \in \hat X_\infty: \xi \leq 0, \eta \geq 0 \text{ in } (0,\infty).
\]
This implies that $u := u_0 + \hu$ and $v := v_0 + \hv$ satisfy in the weak sense the differential inequalities
\begin{align}
u'' + \frac{1}{r}u' - \frac{k^2}{r^2}u &\leq h(u,v),\label{Eq:uDiffIneql}\\
v'' + \frac{1}{r} v' &\geq g(u,v) \text{ in } (0,\infty).\label{Eq:vDiffIneql}
\end{align}

We claim that if $u > 0$ in any interval $(r_1, r_2) \subset (0,\infty)$ then the first equation of \eqref{ODEsystem} (i.e. equality in \eqref{Eq:uDiffIneql}) holds in $(r_1,r_2)$. Indeed, if $\xi \in C_c^\infty(r_1,r_2)$, then in view of the local uniform convergence of $\hu_m$ to $\hu$, there is some $\epsilon_0 > 0$ such that $(\hu_m + t\,\xi, \hv_m)$ belongs to $M_\infty$ for all $|t| < \epsilon_0$ and for all sufficiently large $m$. It thus follows from \eqref{Eq:rho=0} that, there is some $t \in (0,\epsilon_0)$ such that
\[
D\hat\mcE_\infty (\hu_m,\hv_m)(\pm t\xi, 0) \leq \varrho_m\|(t\xi,0)\|_{\hat X_\infty}.
\]
As above, this implies that
\[
D\hat\mcE_\infty (\hu,\hv)(\pm t\xi, 0) \leq 0,
\]
which implies that $D\hat\mcE_\infty (\hu,\hv)( \xi,0) = 0$. Since $\xi$ is arbitrary, the claim follows.

Similarly, if $v < 0$ in any interval $(r_1, r_2) \subset (0,\infty)$, then the second equation of \eqref{ODEsystem} holds in that interval.
 
Since $u$ is continuous, we can write $\{r: u(r) > 0\} = \cup_{j \in \Lambda} (\alpha_j, \beta_j)$ of at most countably many mutually disjoint open intervals. As argued above, the first equation of \eqref{ODEsystem} holds on each interval $(\alpha_j, \beta_j)$. (Initially, it holds in the weak sense, but since $u$ and $v$ are H\"older continuous (in view of the Sobolev embedding theorem in one dimension), it holds in the classical sense.) Furthermore, $u(\alpha_j) = 0$ if $\alpha_j > 0$, and $u(\beta_j) = 0$ if $\beta_j < \infty$. Since $h(u,v) = u\,c_1$ for some continuous function $c_1$, the Hopf lemma implies that 
\begin{equation}
u'(\alpha_j) > 0\text{ if } \alpha_j > 0 \text{ and }u'(\beta_j) < 0 \text{ if }\beta_j < \infty.
\label{Eq:u'<>0}
\end{equation}
Recall that $u' = 0$ a.e. in the set $\{u = 0\}$. Now, for any $\xi \in C_c^\infty(0,\infty)$ and $\xi \geq 0$, we have in view of \eqref{Eq:uDiffIneql} that
\begin{align*}
0
	&\leq \int_0^\infty \big[u'\,\xi' + \frac{1}{r^2} u \xi + h(u,v) \xi\big]\,rdr\\
	&= \sum_{j \in \Lambda} \int_{\alpha_j}^{\beta_j} \big[u'\,\xi' + \frac{1}{r^2} u \xi + h(u,v) \xi\big]\,rdr\\
	&= \sum_{j \in \Lambda} \big[\beta_j\,u'(\beta_j)\,\xi(\beta_j) - \alpha_j\,u'(\alpha_j)\,\xi(\alpha_j)\big],
\end{align*}
where in the first equality, we have used $h(u,v) = 0$ wherever $u = 0$. By \eqref{Eq:u'<>0}, if there is some $j$ such that $\alpha_j$ or $\beta_j$ is non-zero and finite, the last sum is negative if $\xi$ is chosen to be positive thereof. We thus conclude that the $\alpha_j$ and $\beta_j$'s are either zero or infinite, i.e. $u > 0$ in $(0,\infty)$. We hence deduce that the first equation of \eqref{ODEsystem} holds in $(0,\infty)$.

The negativity of $v$ and the validity of the second equation of \eqref{ODEsystem} can be demonstrated similarly, keeping in mind that
\[
g(u,v) \geq v\Big( - a^2 - \frac{1}{\sqrt{6}}b^2\,v + c^2(u^2 + v^2)\Big) =: v\,c_2,
\]
and in particular, $g(u,v) \geq 0$ wherever $v = 0$. We omit the details.
\end{proof}

In the following lemma we prove that $\hat\mcE_R$ satisfies the Palais-Smale condition. 

\begin{lemma} 
For $R \in (0,\infty)$, $\hat \mcE_R$ satisfies the Palais-Smale condition on $\hat X_R$. For $R = \infty$, $\hat\mcE_\infty$ satisfies the Palais-Smale condition on the closed convex set $M_\infty$ defined in \eqref{Eq:MiCone}.
\end{lemma}

\begin{proof}The result is standard for $R < \infty$. Consider the case $R = \infty$. Let $(\hu_m, \hv_m) \in M_\infty$ be a Palais-Smale sequence for $\hat\mcE_\infty$, i.e. $\hat\mcE_\infty(\hu_m, \hv_m)$ is bounded and $\varrho(\hu_m, \hv_m) \rightarrow 0$, where $\varrho$ is defined in \eqref{Eq:varrhoDef}. We need to show that $(\hu_m, \hv_m)$ has a convergent subsequence in $\hat X_\infty$.

By Lemma \ref{Lem:Coercivity}, the sequence $(\hu_m, \hv_m)$ is bounded in $\hat X_\infty$ and so we can assume without loss of generality that $(\hu_m, \hv_m)$ converges weakly in $\hat X_\infty$ to some $(\hu, \hv)$. By the Sobolev embedding theorem (in one and two dimensions), we can also assume that $(\hu_m, \hv_m)$ converges to $(\hu, \hv)$, uniformly on compact subsets of $(0,\infty)$ and strongly in $L^p((0,R);r\,dr)$ for any $R < \infty$ and $1 \leq p < \infty$.

By Lemma \ref{Lem:ConeCrit}, $u_1 : = u_0 + \hu$ and $v_1 := v_0 + \hv$ is a solution to \eqref{ODEsystem}, \eqref{bdrycond}. By working with $(u_1, v_1)$ instead of $(u_0, v_0)$ and with the sequence $(\hu_m - \hu, \hv_m - \hv)$ instead of $(\hu_m, \hv_m)$, we can assume for simplicity that $\hu = \hv = 0$.

Let $$V_m := Df(u_0 + \hu_m, v_0 + \hv_m) - Df(u_0, v_0).$$
As in the proof of Lemma \ref{Lem:Coercivity}, let $u_* = \frac{s_+}{\sqrt{2}}$ and $v_* = -\frac{s_+}{\sqrt{6}}$ and note that $D^2f(u_*,v_*)$ is (strictly) positive definite, which implies that there are $\alpha, \delta > 0$ such that
\[
[Df(x,y) - Df(x',y')] (x - x',y - y') \geq \alpha[(x - x')^2 + (y - y')^2] \quad \forall~(x,y), (x',y') \in B_\delta(u_*,v_*).
\]
Thus in view of \eqref{Eq:uAsymp}, \eqref{Eq:vAsymp} and Strauss' inequality \eqref{Eq:StrIneql} applied to the bounded sequence $(\hu_m, \hv_m)$ in $\hat X_\infty$, there is some large $R_2 > 0$ (independent of $m$) such that
\begin{equation}
V_m  (\hu_m, \hv_m) \geq \alpha[|\hu_m|^2 + |\hv_m|^2] \text{ in } (R_2,\infty).
	\label{Eq:VmLB}
\end{equation}

On the other hand, note that $(t\hu_m,t\hv_m) \in M_\infty$ for all $t \in [0,1]$. As $(\hu_m,\hv_m)$ is bounded in $\hat X_\infty$, we can select some $t_0 \in (0,1)$ independent of $m$ such that
\[
\|(\hu_m - t_0\hu_m, \hv_m - t_0\hv_m)\|_{\hat X_\infty} \leq 1. 
\]
Then it follows from \eqref{Eq:varrhoDef} that
\[
\varrho(\hu_m,\hv_m) \geq D\hat\mcE_\infty (\hu_m, \hv_m) (\hu_m - t_0\hu_m, \hv_m - t_0\hv_m),
\]
which together with \eqref{Eq:VmLB} implies that
\begin{align*}
\frac{1}{1 - t_0}\varrho(\hu_m,\hv_m)
	&\geq D\hat\mcE_\infty(\hu_m,\hv_m)(\hu_m,\hv_m) \\
	&= \int_0^\infty  \Big\{|\hu_m'|^2 + |\hv_m'|^2 + \frac{k^2}{r^2}\,|\hu_m|^2 
		 + V_m  (\hu_m, \hv_m)\Big\}\,r\,dr\\
	&\geq  \int_0^\infty  \Big\{|\hu_m'|^2 + |\hv_m'|^2 + \frac{k^2}{r^2}\,|\hu_m|^2 
		 + \alpha|\hu_m|^2 + \alpha|\hv_m|^2\Big\}\,r\,dr\\
		 	&\qquad + \int_0^{R_2}  \Big\{-\alpha|\hu_m|^2 - \alpha|\hv_m|^2 +  V_m  (\hu_m, \hv_m)\Big\}\,r\,dr.
\end{align*}

On the other hand, by the strong convergence of $(\hu_m, \hv_m)$ to $(0,0)$ in $L^p((0,R_2);r\,dr)$ and the estimate \eqref{Eq:DfEstX2}, we see that
\[
\lim_{m \rightarrow \infty}\int_0^{R_2}  \Big\{-\alpha|\hu_m|^2 - \alpha|\hv_m|^2 +  V_m  (\hu_m, \hv_m)\Big\}\,r\,dr = 0.
\]
Recalling that $\varrho_m \rightarrow 0$, we obtain that $(\hu_m, \hv_m)$ converges in $\hat X_\infty$ to $0 = (\hu, \hv)$.
\end{proof}

The following result is a consequence of the above lemmas and a variant of the mountain pass theorem \cite[Theorem II.12.8]{Struwe}.
\begin{lemma}\label{Lem:MP}
\begin{enumerate}[(a)]
\item Let $R\in (0, \infty)$. If all critical points in $\hat X_R$ of $\hat \mcE_R$ are strictly stable, then $\hat \mcE_R$ has a unique critical point. 
\item Let $R\in (0, \infty]$. If all critical points $(\hat u, \hat v)$ of $\hat \mcE_R$ satisfying $u_0 + \hat u \geq 0$ and $v_0 + \hat v \leq 0$ are strictly stable, then $\hat \mcE_R$ has a unique critical point satisfying $u_0 + \hat u \geq 0$ and $v_0 + \hat v \leq 0$. 
\end{enumerate}
\end{lemma}

In view of the above result, to prove uniqueness in Theorem \ref{Thm:Uniq}, it suffices to establish (strict) stability at relevant critical points $(u,v)$. It is readily seen that, for $R \in (0,\infty]$,
\begin{align*}
D^2\hat\mcE_R(u - u_0, v - v_0) (\xi, \eta)\cdot (\xi, \eta) = B(\xi, \eta) ,
\end{align*}
where $B$ is given by \eqref{Eq:FD2EDef}.

\begin{proposition}\label{Prop:ODEStability}
Let $a^2, b^2, c^2>0$ such that $b^4\leq 3a^2c^2$. Assume that $R \in (0,\infty]$. Let $k\in \ZZ\setminus\{0\}$ and $(u,v)$ be a solution of \eqref{ODEsystem} and \eqref{bdrycond} with $u>0$ and $v<0$. Then $(u,v)$ is strictly stable for $\hat\mcE_R$ in the sense that $B(\xi,\eta)>0$ for every nonzero 
$(\xi,\eta) \in \hat X_R$.
\end{proposition}

\begin{proof} 
We will only prove the case $R = \infty$. (The case $R<\infty$ is simpler since the asymptotical behavior at infinity can be dropped.)

Recall from \cite[Proposition 2.2]{ODE_INSZ}, \eqref{Eq:uAsymp}, \eqref{Eq:vAsymp}, \eqref{Eq:u'AsX}, \eqref{Eq:v'AsX} and \eqref{Eq:u0Behavior}  that
\begin{align}
&u(r) = O(r^{|k|}), u'(r) = O(r^{|k|-1}), v'(r) = O(r) \text{ as } r \rightarrow 0,\label{Eq:uvAs0}\\
&u'(r) = O(r^{-3}), v'(r) = O(r^{-3}) \text{ as } r \rightarrow \infty. \label{Eq:uvAsI}
\end{align}
Recalling $h = \frac{\partial f}{\partial u}$ and $g = \frac{\partial f}{\partial v}$, we obtain the estimate
\begin{align}
&h(u,v) = O(r^{|k|}), g(u,v) = O(1) \text{ as } r \rightarrow 0,\label{Eq:hgAs0}\\
&h(u,v) = O(r^{-2}), g(u,v) = O(r^{-2}) \text{ as } r \rightarrow \infty. \label{Eq:hgAsI}
\end{align}

Fix $(\xi,\eta) \in \hat X_\infty$. Since $u > 0$ and $v < 0$, we can write $\xi =u \tilde\xi$ and $\eta =v \tilde\eta$ where $\tilde \xi, \tilde \eta \in H^1_{loc}(0,\infty)$. By Lemma \ref{Lem:Strauss} and \eqref{Eq:uvAs0}, we have
\begin{align}
&\tilde \xi(r) = o(r^{-|k|}), \tilde \eta(r) = O(|\ln r|^{1/2}) \text{ as } r \rightarrow 0,\label{Eq:txietaAs0}\\
&\tilde \xi(r) = O(r^{-1/2}), \tilde \eta(r) = O(r^{-1/2}) \text{ as } r\rightarrow \infty.\label{Eq:txietaAsI}	
\end{align}

We compute, using \eqref{ODEsystem},
\begin{align*}
\int_{1/m}^m \big[|\xi'|^2 + \frac{k^2}{r^2}\,\xi^2\big]\,r\,dr 
	&=  \int_{1/m}^m \big[ u^2\,|\tilde\xi'|^2 + u'\,(u\,\tilde \xi^2)' + \frac{k^2}{r^2} \,u^2\,\tilde \xi^2\big]\,r\,dr\\
	&=  r\,u'\,u\,\tilde \xi^2\Big|_{1/m}^m + \int_{1/m}^m \big[ u^2\,|\tilde\xi'|^2 - h(u,v)\,u\,\tilde \xi^2\Big]\,r\,dr\\
	&= o(1)+ \int_{1/m}^m \big[ u^2\,|\tilde\xi'|^2 - h(u,v)\,u\,\tilde \xi^2\Big]\,r\,dr \quad \textrm{as} \quad m\to \infty,
\end{align*}
where we have used \eqref{Eq:uvAs0}, \eqref{Eq:uvAsI}, \eqref{Eq:txietaAs0} and \eqref{Eq:txietaAsI} in the last identity. Therefore, by monotone and dominated convergence theorems, and \eqref{Eq:uvAs0}, \eqref{Eq:uvAsI}, \eqref{Eq:hgAs0} and \eqref{Eq:hgAsI}, since $(\xi,\eta) \in \hat X_\infty$, 
\begin{align*}
\int_0^\infty \big[|\xi'|^2 + \frac{k^2}{r^2}\,\xi^2\big]\,r\,dr 
	&= \lim_{m\to \infty}\int_{1/m}^m \big[|\xi'|^2 + \frac{k^2}{r^2}\,\xi^2\big]\,r\,dr =\int_{0}^\infty \big[ u^2\,|\tilde\xi'|^2 - h(u,v)\,u\,\tilde \xi^2\Big]\,r\,dr.
\end{align*}

Likewise, 
\begin{align*}
\int_0^\infty |\eta'|^2\,r\,dr 
	&=  \int_{0}^\infty \big[ v^2\,|\tilde\eta'|^2 - g(u,v)\,v\,\tilde \eta^2\Big]\,r\,dr,
\end{align*}

We hence obtain
\bea
B(\xi, \eta) 
	&= \int_{0}^\infty \bigg\{ u^2| \tilde\xi'|^2 + v^2|\tilde\eta'|^2  \non \\
		&\qquad\qquad+ \left(- \frac{b^2}{v \sqrt{6}} (v^2 +u^2)  + 2c^2 v^2 \right) \eta^2 
 		+ 2c^2 u^2 \xi^2+ 4 u \xi\eta \left( \frac{b^2}{\sqrt{6}} + c^2 v \right) \bigg\}\, rdr.
\eea
Note that $B(\xi, \eta) >  0$ for $(\xi,\eta) \not\equiv 0$, provided that
\bea
 \frac{2\left( \frac{b^2}{\sqrt{6}} + c^2 v \right)^2}{c^2}& \leq \left(- \frac{b^2}{v \sqrt{6}} (v^2 +u^2)  + 2c^2 v^2 \right)\\
\Longleftrightarrow \quad \quad 2b^2\bigg(\frac{b^2}{\sqrt{6}c^2} +v\bigg)+2b^2v &\leq - \frac{b^2}{v} (v^2 +u^2)
\eea
which holds true in $(0, \infty)$ because the above LHS is negative while the RHS is positive due to the inequalities $v<0$ and $b^2+\sqrt{6}c^2v\leq 0$ for $b^4\leq 3a^2c^2$ (see \eqref{Eq:uvBox}).
\end{proof}

We conclude the section with the proof of Theorem \ref{Thm:Uniq}.

\begin{proof}[Proof of Theorem \ref{Thm:Uniq}]
The result is a consequence of Lemma \ref{Lem:MP}(b) and Proposition \ref{Prop:ODEStability}.
\end{proof}

\section{Stability for $k=\pm 1$} \label{sec:stab}

In this section we provide the proof of Theorem~\ref{Thm:Stab} regarding the sign of the second variation $\CL[Q] (P)$ at $k$-radially symmetric solutions $Q$ in direction $P \in H^1_0(B_R, \mcS_0)$. Note that for $R = \infty$, $H^1_0(\RR^2, \mcS_0) \equiv H^1(\RR^2, \mcS_0)$. Recall from \eqref{Eq:CLDef-X} that, for $P \in H_0^1(B_R, \mcS_0)$,
\[
{\CL}[Q](P)
=\int_{B_R}\Big\{|\nabla P|^2-{a^2}|P|^2-2b^2\tr(P^2 Q)+{c^2}\left(|Q|^2|P|^2+2|\tr(QP)|^2\right)\Big\}\,dx.
\]

\subsection{Basis decomposition}\label{ssec:BDecomp}

In order to prove Theorem~\ref{Thm:Stab} we use, as in \cite{INSZ_AnnIHP}, the following basis decomposition. We define $\{ e_i\}_{i=1}^3$ to be the standard basis in $\RR^3$ and denote, for $\varphi \in [0,2\pi)$ and $k \neq 0$,
\[
n=n(\f)=\left(\cos ({\textstyle\frac{k}{2}} \varphi) , \sin ({\textstyle\frac{k}{2}} \varphi) , 0\right), \,
m=m(\f)=\left(-\sin({\textstyle\frac{k}{2}} \varphi),\cos({\textstyle\frac{k}{2}} \varphi),0\right) .
\]
We endow the space $\mcS_0$ of $Q$-tensors with the Frobenius scalar product 
$$Q\cdot \tilde Q=\tr(Q\tilde Q)$$
and 
for any $\f\in [0, 2\pi)$, we define the following orthonormal basis in $\mcS_0$:
\begin{align*}
E_0&=\sqrt{\frac{3}{2}}\left(e_3\otimes e_3-\frac{1}{3}I_3\right),\\
E_1&=E_1(\f)=\sqrt{2}\left(n\otimes n-\frac{1}{2}I_2\right),\,E_2=E_2(\f)=\frac{1}{\sqrt{2}}\left(n\otimes m+m\otimes n\right), \non \\
E_3&=\frac{1}{\sqrt{2}}(e_1 \otimes e_3+e_3\otimes e_1),\, E_4=\frac{1}{\sqrt{2}}\left(e_2\otimes e_3+e_3\otimes e_2\right).
\end{align*}
Obviously, only $E_1$ and $E_2$ depend on $\f$ and we have 
\be
\label{deriv_E12}
\frac{\partial E_1}{\partial \f}=kE_2 \quad \textrm{and} \quad  \frac{\partial E_2}{\partial \f}=-kE_1.
\ee

The above basis $\{E_0, \dots, E_4\}$ is constructed so that at a point $Q_*=s_+\left(n\otimes n-\frac{1}{3}I_3\right)$ with $n\in\mathbb{S}^1\times \{0\}$,  tensor $E_2$ is along the direction of the tangent line to $\mcS^{lim}_*$ (see \eqref{s-lim}),  while tensors  $E_0$, $E_1$, $E_3$ and $E_4$ are the normal directions to the tangent line.

It is clear that any $P \in  H_{loc}^1(\R^2, \mcS_0)$ can be represented as
$$
P(x)= \sum_{i=0}^4 w_i(x) E_i, \quad x = r(\cos\varphi, \sin\varphi) \in \R^2,
$$
with $w_i=P\cdot E_i$ for $i=0,\dots, 4$. We note although $n$ and $m$ may not be smooth as a function of $x$, $E_i$ are smooth away from the origin. Then the second variation becomes
\begin{align}
\CL [Q](P) &= \int_{B_R} \bigg\{\sum_{i=0}^4 |\nabla w_i|^2+\frac{k^2}{r^2}(w_1^2+w_2^2)+\frac{2k}{r^2}\left(
w_1 \frac{\partial w_2}{\partial \f}-w_2 \frac{\partial w_1}{\partial \f}\right) \non \\
&+ \left(-{a^2} + c^2 (u^2+v^2) \right)\sum_{i=0}^4 |w_i |^2+2 {c^2}\left( v w_0 + u w_1 \right)^2 \non\\
 &-\frac{2b^2}{\sqrt{6}} \left( v (  w_0^2 - w_1^2 - w_2^2) - 2 u w_0 w_1\right)\non\\
 &-\frac{2b^2}{\sqrt{6}}\bigg(\frac{\sqrt{3}}{2}u(w_3^2-w_4^2)\cos(k\f)+\sqrt{3}uw_3w_4\sin(k\f)+\frac{1}{2}v(w_3^2+w_4^2) \bigg)\bigg\}\, dx.\label{Eq:7I2016-1}
\end{align}
We note that components $\{w_0, w_1, w_2\}$ and $\{ w_3, w_4\}$ in \eqref{Eq:7I2016-1} are not mixed and therefore we can separately study the sign of $\CL [Q](P)$ in the spaces
\bea
&V_1 = \{ P \in H_0^1(B_R, \mcS_0) \, : \, P \cdot E_3 = P \cdot E_4 =0 \},\\
&V_2 = \{ P \in H_0^1(B_R, \mcS_0) \, : \, P = w_3 (x) E_3 + w_4(x) E_4; \, w_3, w_4 \in H_0^1(B_R) \}.
\eea
It is clear that $H_0^1(B_R, \mcS_0) =V_1 \oplus V_2 $. Furthermore, if $P$ belongs to $H^1_0(B_R, \mcS_0)$, then so do its (direct sum) projections onto $V_1$ and $V_2$.

\subsection{Stability in the space $V_1$}

We start with the result about stability of $\mcL[Q]$ in $V_1$.

\begin{proposition}\label{Prop:V1Stab}
Let  $a^2 \geq 0$, $b^2, c^2 >0$ be fixed constants, $R \in (0, \infty]$, and $k =\pm 1$. Let $(u,v)$ be a solution of \eqref{ODEsystem} on $(0,R)$ under the boundary condition \eqref{bdrycond} such that $u > 0$ and $v < 0$ and assume that $(u,v)$ is stable with respect to $\mcE_R$ (i.e. \eqref{Eq:ODEStab} holds). Let $Q$ be $k$-radially symmetric solution $Q$ of \eqref{eq:EL} (on $B_R$) and \eqref{BC1} given by \eqref{anY}. Then $\CL [Q](P) \geq 0$ for any $P \in V_1$.
\end{proposition}

We will use the following lemma whose simple proof we omit.

\begin{lemma}\label{lem:wispace}
Let $|k| = 1$ and $P = w_0\,E_0 + w_1\,E_1 + w_2\,E_2 \in V_1$. If we write 
\begin{equation}
w_l (r,\f) = \sum_{m=-\infty}^\infty (w_{l,m}(r) + i\,\hat w_{l,m}(r)) e^{i m \f} \quad \textrm{ with }\, \,  l=0,1,2,
	\label{Eq:7I2016-2X}
\end{equation} 
Then
\begin{align*}
&\sqrt{r} (|w_{l,m}'| + |w_{l,m}|), \sqrt{r} (|\hat w_{l,m}'| + |\hat w_{l,m}|)\in L^2(0,R) \text{ for all } (l,m),\\
& \frac{1}{\sqrt{r}} |w_{l,m}|, \frac{1}{\sqrt{r}} |\hat w_{l,m}| \in L^2(0,R) \text{ for all } (l,m) \notin \{(0,0), (1,\pm1), (2,\pm 1)\},\\
&\frac{1}{\sqrt{r}} |kw_{1,m} - m \hat w_{2,m}|, \frac{1}{\sqrt{r}} |k w_{2,m} + m \hat w_{1,m}| \in L^2(0,R) \text{ for } |m| = 1.
\end{align*}
Furthermore, for each $m \in \ZZ$, 
\[
\sum_{l = 0}^2 (w_{l,m}(r) + i\,\hat w_{l,m}(r)) e^{i m \f}\,E_l \in V_1.
\]
\end{lemma}

\begin{proof}[Proof of Proposition \ref{Prop:V1Stab}] Let us first show that $\mcL[Q](P) \geq 0$ for all $P \in V_1$. By standard density argument, we can assume without loss of generality that $P \in V_1 \cap C_c^\infty(B_R \setminus \{0\})$. (Here we have used the fact that a point has zero Newtonian capacity in two dimensions.)

We write $x=re^{i\f}=(r\cos\f, r\sin \f)$ and $P = w_0 E_0 + w_1 E_1 + w_2 E_2$ as in subsection \ref{ssec:BDecomp}. By \eqref{Eq:7I2016-1},
\bea
\CL [Q](P) &= \int_0^R \int_0^{2\pi} \bigg\{\sum_{l=0}^2 |\partial_r w_l |^2 +\frac{1}{r^2} \left( |\partial_\f w_0|^2+ |\partial_\f w_1 - k w_2|^2 + |\partial_\f w_2 + k w_1|^2 \right) \non \\
&+ \left(-{a^2} + c^2 (u^2+v^2) \right)\sum_{l=0}^2 |w_l |^2+2 {c^2}\left( v w_0 + u w_1 \right)^2 \\
 &-\frac{2b^2}{\sqrt{6}} \left( v (  w_0^2 - w_1^2 - w_2^2) - 2 u w_0 w_1 \right)\bigg\}\, rdr\,d\f.
\eea

Now, we Fourier decompose $w_l$'s as in \eqref{Eq:7I2016-2X}. By Lemma \ref{lem:wispace}, 
\begin{equation}
P_m := \sum_{l = 0}^2 (w_{l,m}(r) + i\,\hat w_{l,m}(r)) e^{i m \f}\,E_l \in V_1.
	\label{Eq:7I2016-2BIS}
\end{equation}
Furthermore, a direct computation shows that
\begin{equation}
\CL [Q](P)= \sum_{m=-\infty}^\infty \CL[Q](P_m),
	\label{Eq:CLSEP}
\end{equation}
and
\begin{align*}
\CL [Q](P_m)
	&= 2\pi\int_0^R \Biggl\{ \sum_{l=0}^2 [| w_{l,m}'|^2+|\hat w_{l,m}'|^2] + \frac{m^2}{r^2}[|w_{0,m}|^2 +|\hat w_{0,m}|^2]  \non \\
		&\qquad + \frac{1}{r^2} \Big[| m w_{1,m} - k \hat w_{2,m}|^2  +  | k w_{1,m} - m \hat w_{2,m}|^2\non\\
			&\qquad\qquad +| m \hat w_{1,m} + k w_{2,m}|^2  + | k \hat w_{1,m} + m w_{2,m}|^2  \Big] \non \\ 
   		&\qquad + \left(-a^2 -\frac{2}{\sqrt{6}}b^2 v+c^2\left( u^2+ 3 v^2\right) \right) (|w_{0,m}|^2 +|\hat w_{0,m}|^2)\non\\
		&\qquad + \left(-a^2 +\frac{2}{\sqrt{6}}b^2 v+c^2\left(3 u^2+  v^2\right) \right) (|w_{1,m}|^2 +|\hat w_{1,m}|^2)\\
		&\qquad + 4 u (w_{0,m}  w_{1,m} + \hat w_{0,m} \hat w_{1,m})  \left[ \frac{b^2}{\sqrt{6}} + c^2 v \right]  \non \\
		&\qquad+ \left(-a^2 + \frac{2}{\sqrt{6}}b^2 v+c^2\left( u^2+ v^2\right) \right) (| w_{2,m} |^2 + | \hat w_{2,m} |^2)  \Biggr\} r\, dr. \non
\end{align*}
Now, observe that if we define
\begin{align*}
\TCL_m (w_0, w_1, w_2)
	& = \int_0^R  \Biggl\{ | w_0'|^2 +  | w_1'|^2  + |  w'_2 |^2 + \frac{m^2 }{r^2}|w_0|^2\\
		&\qquad  + \frac{1}{r^2} \left( \big| |m| w_1 - |k| w_2\big|^2  + \big| |k| w_1 - |m|  w_2\big|^2  \right)  \\
		&\qquad+ \left(-a^2 -\frac{2}{\sqrt{6}}b^2 v+c^2\left( u^2+ 3 v^2\right) \right) |w_0|^2\non \\
		&\qquad + \left(-a^2 +\frac{2}{\sqrt{6}}b^2 v+c^2\left( 3u^2+ v^2\right) \right) |w_1|^2
			  +  4 u  w_0  w_1  \left[ \frac{b^2}{\sqrt{6}} + c^2 v \right]  \non \\
		&\qquad + \left(-a^2 + \frac{2}{\sqrt{6}}b^2 v+c^2\left( u^2+ v^2\right) \right)  |  w_2 |^2  \Biggr\} r\, dr, \non
\end{align*}
then
\begin{multline*}
\frac{1}{2\pi}\CL [Q](P_m)=\TCL_m(\sgn(m)w_{0,m}, \sgn(m) w_{1,m}, \,  \sgn(k)\hat w_{2,m}) \\+\TCL_m(\sgn(m)\hat w_{0,m}, \sgn(m)\hat w_{1,m}, -\sgn(k)w_{2,m}),
\end{multline*}
where we use the convention that $\sgn(0) = 1$.

From the foregoing analysis, in order to show that $\mcL$ is non-negative on $V_1$, it is enough to show that $\TCL_m (w_0, w_1, w_2) \geq 0$ for any smooth functions $w_0, w_1, w_2 \in C_c^\infty(0,R)$. In addition, it is clear from the definition of $\TCL_m$ that it suffices to consider $k = 1$ and $m\geq 0$. We consider the cases $m \geq 1$ and $m = 0$ separately.

\vskip 0.2cm

\nd {\bf I. Case $m\geq 1$}: Consider first the case $b^4 \neq 3a^2c^2$. In this case, since $u' > 0$, $v' \neq 0$ (see Theorem \ref{Thm:Mon}) and $u > 0$, we can write $w_0 = v' \eta$, $w_1=u' \xi$, $w_2 = u \zeta$ for $\eta, \xi, \zeta \in C_c^\infty(0,R)$ and use Hardy decomposition trick to obtain
\bea
\TCL_m (w_0, w_1, w_2)& = \int_0^R  \Biggl\{ | v'|^2 |\eta'|^2 +(m^2-1)\frac{(v')^2\eta^2}{r^2}- 2 u  v'  u' (\xi - \eta)^2  \left[ \frac{b^2}{\sqrt{6}} + c^2 v \right]  \non \\
&\qquad+ |u'|^2 |\xi'|^2 + \frac{1}{r^2} \left( | m u' \xi -  u \zeta |^2  + | u' \xi - m  u \zeta|^2  \right) \non \\ 
&\qquad  -\frac{2}{r^2} |u'|^2 \xi^2 + \frac{2}{r^3} uu' \xi^2+  |u|^2 |\zeta'|^2-\frac{1}{r^2} |u|^2 \zeta^2 \Biggr\} r\, dr\\
&= J_m + I_m,
\eea 
where
\[
J_m
	=\int_0^R  \Biggl\{ | v'|^2 |\eta'|^2 + |u'|^2 |\xi'|^2 +(m^2-1)\frac{|v'|^2\eta^2}{r^2} - 2 u  v'  u' \left[ \frac{b^2}{\sqrt{6}} + c^2 v \right] (\xi -\eta)^2 \Biggr\} r\, dr,
\]
\begin{align*}
I_m
	= \int_0^R  \Biggl\{ \frac{m^2 -1}{r^2} |u'|^2 \xi^2 +  \frac{m^2 }{r^2} |u|^2 \zeta^2 +  |u|^2 |\zeta'|^2 + \frac{2}{r^3} uu'  \xi^2 - \frac{4m}{r^2} u u' \xi \zeta\Biggr\} r\, dr.
\end{align*}
Since $v' \big(\frac{b^2}{\sqrt{6}} + c^2 v\big) < 0$ in $(0, R)$, it is clear that $J_m \geq 0$. As for $I_m$, we compute
\bea
I_m
&= \int_0^R  \Biggl\{\frac{m^2 -1}{r^2} \left( |u'|^2 \xi^2 + |u|^2 \zeta^2 \right) - \frac{4(m-1)}{r^2} u u' \xi \zeta \\
&\quad \quad +  \frac{1 }{r^2} |u|^2 \zeta^2 +  |u|^2 |\zeta'|^2 + \frac{2}{r^3} uu'  \xi^2 - \frac{4}{r^2} u u' \xi \zeta \non \Biggr\} r\, dr\\
&=\int_0^R  \Biggl\{ \frac{(m-1)^2}{r^2} \left( |u'|^2 \xi^2 + |u|^2 \zeta^2 \right) +    \frac{2(m-1)}{r^2} (u' \xi -u \zeta)^2  \non \\
& \quad \quad  + \frac{2u u' }{r^3} \left(  \xi - \zeta r\right)^2 + \frac{|u|^2 }{r^2} \left( \zeta +\zeta' r \right)^2\Biggr\} r\, dr
\geq 0.
\eea
(Here we have used the identity $\int_0^R \big(u^2\zeta \zeta'+uu'\zeta^2\big)\, dr= \frac{1}{2}\int_0^R (u^2\zeta^2)'\,dr = 0$.) We conclude that $\TCL_m (w_0, w_1, w_2) \geq 0$ for $m \geq 1$ and $b^4 \neq 3a^2c^2$.

Let us now turn to the case $b^4 = 3a^2c^2$. By \cite[Proposition 3.5]{INSZ_AnnIHP}, $v \equiv -\frac{s_+}{\sqrt{6}}$ and $u$ is the unique solution of
\[
u'' + \frac{1}{r}\,u' - \frac{k^2}{r^2}\,u = c^2u(u^2 - \frac{s_+^2}{2}), u(0) = 0, u(R) = \frac{s_+}{\sqrt{2}}.
\]
Furthermore $u > 0$ and $u' > 0$. 

The argument above for $b^4 \neq 3a^2c^2$ does not apply directly since $v' \equiv 0$ and we cannot write $w_0 = v'\,\eta$ unless $w_0 \equiv 0$. Nevertheless, with the above explicit value of $v$, the expression for $\TCL_m$, for $m \geq 0$, simplifies to 
\begin{align*}
\TCL_m (w_0, w_1, w_2)
	& = \int_0^R  \Biggl\{ | w_0'|^2 +  | w_1'|^2  + |  w'_2 |^2 + \frac{m^2 }{r^2}|w_0|^2
		  + \frac{1}{r^2} \left( \big| m w_1 - w_2\big|^2  + \big| w_1 - m  w_2\big|^2  \right)  \\
		&\qquad+ \frac{1}{2}c^2 (2u^2 + s_+^2) |w_0|^2 + \frac{1}{2}c^2( 6u^2 - s_+^2) |w_1|^2 + \frac{1}{2}c^2 (u^2 - s_+^2)  |  w_2 |^2  \Biggr\} r\, dr.
\end{align*}
It is readily seen that the contribution of $w_0$ is non-negative and uncoupled with $w_1$ and $w_2$. Thus, in proving the positivity of $\TCL_m$, we can assume without loss of generality that $w_0 \equiv 0$. The foregoing analysis now applies yielding $\TCL_m(w_0, w_1, w_2) \geq 0$ for all $m \geq 1$.

\bigskip

\nd {\bf II. Case $m=0$}: Note that
$$
\TCL_0(w_0, w_1, w_2)=B(w_1, w_0)+\tilde{F}(w_2)$$
where $B$ stands for the second variation of $\mcE_R(u, v)$ (see \eqref{Eq:FD2EDef}) while
\begin{align*}
\tilde{F}(w_2)&=\int_0^R \bigg\{|w'_2 |^2 +\frac{w_2^2}{r^2}+ \left(-a^2 + \frac{2}{\sqrt{6}}b^2 v+c^2\left( u^2+ v^2\right) \right)  |  w_2 |^2  \Biggr\} r\, dr\non\\
&=\int_0^R (\zeta')^2u^2 \, rdr  \geq 0
\end{align*}
by the computation in the previous case with the Hardy decomposition $w_2=u\zeta$. One concludes that $\TCL_0(w_0, w_1, w_2) \geq 0$ thanks to \eqref{Eq:ODEStab}.
\end{proof}

Let us now turn to the study of the kernel of $\mcL[Q]$ in $V_1$.

\begin{proposition}\label{Prop:V1Ker}
Under the hypotheses of Proposition \ref{Prop:V1Stab}, $\CL[Q](P) = 0$ for some $P \in V_1$ if and only if the dichotomy in the second part of Theorem \ref{Thm:Stab} holds.
\end{proposition}

\begin{proof}
We will consider only the case $k = 1$ and omit the very similar proof for $k = -1$.

Assume that $P \in V_1$ and $\CL[Q](P) = 0$. Define $P_m$ as in \eqref{Eq:7I2016-2BIS} so that $P = \sum P_m$. By \eqref{Eq:CLSEP} and Proposition \ref{Prop:V1Stab}, we have that $\CL[Q](P_m) = 0$ for all $P_m$.

Define
\begin{equation}
Y = \Big\{w: (0,R) \rightarrow \RR\Big| \sqrt{r}(|w'| + |w|)\in L^2(0,R), \text{ and }w(R) = 0 \text{ if } R < \infty\Big\}.
	\label{Eq:YDEF}
\end{equation}
For the functionals $\TCL_m$ defined in the proof of Proposition \ref{Prop:V1Stab}, we make the following four claims.
\begin{enumerate}[(i)]
\item For all $m \geq 2$ and $w_l \in Y \cap L^2((0,R);\frac{1}{r}\,dr)$, there holds
$$\TCL_m(w_0, w_1, w_2) 
	\geq \int_0^R \frac{1}{r} \sum_{l = 0}^2 |w_{l}|^2\,dr.
$$
\item For $b^4 \neq 3a^2c^2$ and for all $w_0 \in Y \cap L^2((0,R);\frac{1}{r}\,dr)$ and $w_1, w_2 \in Y$ such that $|w_1 - w_2| \in L^2((0,R);\frac{1}{r}\,dr)$, there holds
\begin{align*}
\TCL_1(w_0, w_1, w_2) 
	&\geq \int_0^R \bigg\{ \frac{|u|^2 }{r^2} |(r\zeta)'|^2 -2uu'v' \left[ \frac{b^2}{\sqrt{6}} + c^2 v \right] (\xi - \eta)^2  + \frac{2u u' }{r^3} \left(  \xi - \zeta r\right)^2\Biggr\} r\, dr ,
\end{align*}
where $\eta = \frac{w_0}{v'}, \xi = \frac{w_1}{u'}$ and $\zeta = \frac{w_2}{u}$.
\item For $b^4 = 3a^2c^2$ and for all $w_0 \in Y \cap L^2((0,R);\frac{1}{r}\,dr)$ and $w_1, w_2 \in Y$ such that $|w_1 - w_2| \in L^2((0,R);\frac{1}{r}\,dr)$, there holds
\begin{align*}
\TCL_1(w_0, w_1, w_2) 
	&\geq \int_0^R \bigg\{\frac{1}{r^2} |w_0|^2 + \frac{|u|^2 }{r^2} |(r\zeta)'|^2   + \frac{2u u' }{r^3} \left(  \xi - \zeta r\right)^2\Biggr\} r\, dr ,
\end{align*}
where $\xi = \frac{w_1}{u'}$ and $\zeta = \frac{w_2}{u}$.
\item For all $w_0, w_2 \in Y \cap L^2((0,R);\frac{1}{r}\,dr)$ and $w_1 \in Y$, there holds
$$
\TCL_0(w_0, w_1, w_2) 
	= B(w_1, w_0) + \int_0^R  \frac{(w_2'u - w_2\,u')^2}{u^2}\,r\,dr.
$$
\end{enumerate}
When $w_0, w_1, w_2 \in C_c^\infty(0,R)$, the above claims were established in the proof of Proposition \ref{Prop:V1Stab}. They continue to hold in this generality, thanks to Fatou's lemma, since the left hand sides are quadratic linear forms while the integrands on the right hand sides are non-negative.

Now, we see that $\TCL_m(w_0, w_1, w_2) = 0$ if and only if one of the following three cases occurs:
\begin{itemize}
\item $m \geq 2$ and $w_0 = w_1 = w_2 = 0$,
\item or $m = 0$ and $w_2 = 0$ and $B(w_1, w_0) = 0$,
\item or $m = 1$ and $(w_0, w_1, w_2) = (tv', tu', \frac{t}{r}u)$ for some constant $t$.
\end{itemize}
The conclusion is then readily seen from the above and the fact that
\begin{align*}
\frac{\partial Q}{\partial x_1}&=v'(r)\cos \f E_0+u'(r)\cos \f E_1-k\frac{u(r)}{r}\sin \f E_2,\\
\frac{\partial Q}{\partial x_2}&=v'(r)\sin \f E_0+u'(r)\sin \f E_1+k\frac{u(r)}{r}\cos \f E_2. 
\end{align*}
We omit the details.
\end{proof}

\subsection{Stability in the space $V_2$}

\begin{proposition}\label{Prop:V2Stab}
Let  $a^2 \geq 0$, $b^2, c^2 >0$ be fixed constants, $R \in (0, \infty]$, and $k =\pm 1$. Let $(u,v)$ be a solution of \eqref{ODEsystem} on $(0,R)$ under the boundary condition \eqref{bdrycond} such that $u > 0$ and $v < 0$ and let $Q$ be $k$-radially symmetric solution $Q$ of \eqref{eq:EL} (on $B_R$) and \eqref{BC1} given by \eqref{anY}. Then $\CL [Q](P) > 0$ for all nonzero $P \in V_2$. 
\end{proposition}

\begin{proof}
We will consider only the case $k = 1$ and omit the very similar proof for $k = -1$.

Let $P=w_3E_3+w_4 E_4$. Then 
\bea
\CL [Q](P) &= \int_0^R \int_0^{2\pi}\bigg\{ \sum_{i=3}^4 \bigg[|\partial_r w_i |^2 +\frac{1}{r^2}  |\partial_\f w_i|^2 + \left(-{a^2} -\frac{b^2}{\sqrt{6}} v + c^2 (u^2+v^2)  \right) |w_i|^2 \bigg]\\
 &-\frac{b^2 u}{\sqrt{2}}\left((w_3^2-w_4^2)\cos(k\f)+2uw_3w_4\sin(k\f) \right)\bigg\}\, rdrd\f.
\eea

We will represent $$w= w_3 + i w_4$$ to obtain
$$
\CL [Q](P) = \int_0^R \int_0^{2\pi}\bigg\{  |\nabla w |^2  + \left(-{a^2} -\frac{b^2}{\sqrt{6}} v + c^2 (u^2+v^2)  \right) |w|^2 -\frac{b^2 u}{\sqrt{2}} {\rm Re} (w^2 e^{-i\f})\bigg\}\, rdrd\f.
$$
Now we can use Fourier decomposition
$$
w = \sum_{n \in \ZZ} \xi_n (r) e^{i n\f}.
$$
We note that $\xi_n \in Y \cap L^2((0,R);\frac{1}{r}dr)$ for $n \neq 0$ and $\xi_0 \in Y$, where $Y$ is defined by \eqref{Eq:YDEF}.

It is clear that
$$
\int_0^{2\pi} w^2 e^{-i \f} \, d\f = \int_0^{2\pi} \sum_{n,m \in \ZZ} \xi_n (r) \xi_m (r) e^{i (n+m-1)\f} \, d\f = 4\pi \sum_{n=1}^\infty \xi_n (r) \xi_{1-n} (r).
$$
Therefore
\bea
\frac{\CL [Q](P)}{2 \pi} &= \sum_{n=1}^\infty \int_0^R \left[ |\xi_n'|^2 +  |\xi_{1-n}'|^2 + \frac{n^2}{r^2} |\xi_n |^2 +  \frac{(1-n)^2}{r^2} |\xi_{1-n} |^2 \right.\non \\
&+ \left. \left(-{a^2} -\frac{b^2}{\sqrt{6}} v + c^2 (u^2+v^2)  \right) (|\xi_n|^2 + |\xi_{1-n}|^2) - \sqrt{2} b^2 u {\rm Re} (\xi_n \xi_{1-n} ) \right] r\, dr\\
&=J_1+J_2,
\eea
where $J_1$ and $J_2$ correspond to $n=1$ and $n \geq 2$. 

\medskip

\nd {\bf Estimating $J_2$}. We use Hardy decomposition trick $\xi_n = u \eta_n$ for $n\geq 2$ and $n \leq -1$ to obtain
\bea
J_2 &= \sum_{n=2}^\infty \int_0^R \left[ |\xi_n'|^2 +  |\xi_{1-n}'|^2 + \frac{n^2}{r^2} |\xi_n |^2 +  \frac{(1-n)^2}{r^2} |\xi_{1-n} |^2 \right.\non \\
		&\qquad\qquad+ \left. \left(-{a^2} -\frac{b^2}{\sqrt{6}} v + c^2 (u^2+v^2)  \right) (|\xi_n|^2 + |\xi_{1-n}|^2) - \sqrt{2} b^2 u {\rm Re} (\xi_n \xi_{1-n} ) \right] r\, dr\\
&= \sum_{n=2}^\infty \int_0^R \left[ |\eta_n'|^2 +  |\eta_{1-n}'|^2 + \frac{n^2-1}{r^2} |\eta_n |^2 +  \frac{(1-n)^2-1}{r^2} |\eta_{1-n} |^2 \right.\non \\
		&\qquad\qquad \left. -\frac{3 b^2}{\sqrt{6}} v (|\eta_n|^2 + |\eta_{1-n}|^2) 
		 - \sqrt{2} b^2 u {\rm Re} (\eta_n \eta_{1-n} ) \right] u^2 r\, dr\\
 &\geq \sum_{n=2}^\infty \int_0^R \frac{b^2}{\sqrt{2}}(-\sqrt{3} v - u)(|\eta_n|^2 + |\eta_{1-n}|^2) u^2 r\, dr\\
&\geq \sum_{n=2}^\infty \int_0^R \frac{b^2}{\sqrt{2}}(-\sqrt{3} v - u)(|\xi_n|^2 + |\xi_{1-n}|^2) r\, dr
\eea
Using the fact that $-\sqrt{3} v -u > 0$ in $(0,R)$ (cf. \eqref{Eq:r3vu}), we obtain that 
$J_2>0$ for nonzero modes $\{\xi_n\}_{n\neq 0,1}$. (Strictly speaking the above estimates are first shown for $\xi_n \in C_c^\infty(0,R)$ and then extend to $\xi_n \in Y \cap L^2((0,R);\frac{1}{r}\,dr)$ by density.)

\medskip

\nd {\bf Estimating $J_1$}. For $\xi_0, \xi_1 \in C_c^\infty(0,R)$, we have by Hardy decomposition trick for $\xi_0 = v \eta_0$ and $\xi_1 = u \eta_1$:
\bea
J_1 &= \int_0^R \left[ |\xi_1'|^2 +  |\xi_{0}'|^2 + \frac{1}{r^2} |\xi_1 |^2  \right.\non \\
	&\qquad+ \left. \left(-{a^2} -\frac{b^2}{\sqrt{6}} v + c^2 (u^2+v^2)  \right) (|\xi_1|^2 + |\xi_{0}|^2) - \sqrt{2} b^2 u {\rm Re} (\xi_1 \xi_{0} ) \right] r\, dr.\\
& =  \int_0^R \left[ |\eta_1'|^2 u^2 +  |\eta_0'|^2 v^2 - \frac{3 b^2}{\sqrt{6}} v  |\xi_1|^2 - \frac{b^2}{\sqrt{6} v} u^2  |\xi_0|^2
 - \sqrt{2} b^2 u {\rm Re} (\xi_1 \xi_0 )  \right] r\, dr\\
& =  \int_0^R \left[ |\eta_1'|^2 u^2 +  |\eta_0'|^2 v^2 - \frac{b^2}{v \sqrt{6}} \left| \sqrt{3} v \xi_1 + u \xi_0 \right|^2 \right] r\, dr.
\eea
As in the proof of Proposition \ref{Prop:V1Ker}, this leads to\footnote{Equality can actually be shown using precise asymptotical behaviors of $u$ and $v$ at the origin and at infinity, but this weaker form suffices for our purpose here.}
\[
J_1 
\geq  \int_0^R \left[ |\eta_1'|^2 u^2 +  |\eta_0'|^2 v^2 - \frac{b^2}{v \sqrt{6}} \left| \sqrt{3} v \xi_1 + u \xi_0 \right|^2 \right] r\, dr
\]
for $\xi_0 \in Y$ and $\xi_1 \in Y \cap L^2((0,R);\frac{1}{r}\,dr)$.
Therefore, $J_1>0$ for nonzero modes $\{\xi_n\}_{n=0,1}$. We conclude that $\CL [Q](P) >0$.
\end{proof}

\subsection{Proof of Theorem \ref{Thm:Stab}}

The theorem is a consequence of Propositions \ref{Prop:V1Stab}, \ref{Prop:V1Ker} and \ref{Prop:V2Stab}.\hfill$\Box$

\section*{Acknowledgment.} The authors gratefully acknowledge the hospitality and partial support of the Centre International de Rencontres Math\'{e}matiques, Institut Henri Poincar\'{e}, and Centro di Ricerca Matematica Ennio De Giorgi where parts of this work were carried out. R.I. acknowledges partial support by the ANR project ANR-14-CE25-0009-01. V.S. acknowledges partial support by EPSRC grant EP/K02390X/1. V.S. and A.Z. acknowledge partial support of Leverhulme Research Grant RPG-2014-226. The activity of A.Z. on this work was partially supported by a grant of the Romanian National Authority for Scientific Research and Innovation, CNCS-UEFISCDI, project number PN-II-RU-TE-2014-4-0657.

\appendix
\section{Appendix}

\begin{lemma}\label{Lem:CalcLem}
Assume that $a^2 \geq 0, b^2, c^2 > 0$. Let
\[
f(x,y) = -\frac{a^2}{2}(x^2 + y^2) + \frac{c^2}{4}(x^2 + y^2)^2 - \frac{b^2}{3\sqrt{6}}y(y^2 - 3x^2).
\]
Then
\[
\min_{\RR^2} f = -\frac{a^2}{3} s_+^2 - \frac{2b^2}{27}\,s_+^3 + \frac{c^2}{6} s_+^4 
\]
which is attained at (and only at) $\big(0,\frac{2}{\sqrt{6}}s_+\big)$ and $\big(\pm\frac{1}{\sqrt{2}}s_+,- \frac{1}{\sqrt{6}}s_+\big)$. Furthermore, the Hessian of $f$ at all these critical points is positive definite.
\end{lemma}

\begin{proof} We write $x = r\,\sin\varphi$ and $y = r\cos \varphi$ for some $r \geq 0$ and $\varphi \in [0,2\pi)$. Then
\[
f(x,y) = -\frac{a^2}{2}r^2 + \frac{c^2}{4}r^4 - \frac{b^2}{3\sqrt{6}}\,r^3 \cos 3\varphi \geq -\frac{a^2}{2}r^2 + \frac{c^2}{4}r^4 - \frac{b^2}{3\sqrt{6}}\,r^3 =: \tilde f(r).
\]
It is easy to check that $\tilde f$ has three critical points, $r = 0$ and $r = \frac{2}{\sqrt{6}}s_\pm$ where the first one is a local maximum point and the other two are local minimum points. The global minimum of $\tilde f$ is then verified to achieved at $r = \frac{2}{\sqrt{6}}s_+$. We have thus shown that 
\[
f(x,y) \geq \tilde f\big(\frac{2}{\sqrt{6}}s_+\big) = -\frac{a^2}{3} s_+^2  - \frac{2b^2}{27}\,s_+^3 + \frac{c^2}{6} s_+^4 ,
\]
and equality is attained if  and only if $r =\frac{2}{\sqrt{6}}s_+$ and $\varphi \in \{0, \frac{2\pi}{3}, \frac{4\pi}{3}\}$. The first assertion follows.

Now a computation using $-a^2 - \frac{b^2}{3} s_+ + \frac{2}{3}c^2\,s_+^2 = 0$ leads to
\begin{align*}
D^2 f\big(0,\frac{2}{\sqrt{6}}s_+\big) 
	&= \left[\begin{array}{cc}b^2\,s_+ & 0\\0& \frac{1}{3}(3a^2 + b^2\,s_+) \end{array}\right],\\
	D^2 f\big(\pm \frac{1}{\sqrt{2}}s_+,- \frac{1}{\sqrt{6}}s_+\big) 
	&= \left[\begin{array}{cc}c^2 s_+^2 & \pm \frac{1}{\sqrt{3}}(-c^2\,s_+^2+ b^2\,s_+)\\\pm \frac{1}{\sqrt{3}}(-c^2\,s_+^2+ b^2\,s_+) & \frac{1}{3}(c^2\,s_+^2 + 2b^2\,s_+) \end{array}\right],
\end{align*}
from which the last assertion follows.
\end{proof}

\bibliographystyle{acm}
\bibliography{paris,LiquidCrystals}

\end{document}